\documentclass[11pt]{article}
\usepackage{amsmath, amsthm, amssymb}
\usepackage{amscd,amssymb}
\usepackage[all]{xy}
\usepackage{color}
\usepackage{comment}
\usepackage{appendix}
\usepackage{verbatim}
\usepackage{mathrsfs}
\usepackage{pdfsync}
\usepackage{graphicx,epsfig}
\newtheorem{thm}{Theorem}[section]

\newtheorem{lem}[thm]{Lemma}
\usepackage{tikz}
\usetikzlibrary{matrix}
\usepackage{hyperref}

\usepackage{tikz}
\usetikzlibrary{shapes,arrows}
\usepackage{caption}
\usepackage{subcaption}
\usetikzlibrary{positioning}
\usepackage{indentfirst}
\usepackage[T1]{fontenc}
\usepackage[ansinew]{inputenc}
\usepackage{float}
\usetikzlibrary{shapes,arrows}
\usepackage{setspace}
\usepackage{lipsum}

\date{}

\let\oldproofname=\proofname
\renewcommand{\proofname}{\rm\bf{\oldproofname}}

 \newtheorem{prop}[thm]{Proposition}
  
\theoremstyle{definition}
 
 \newtheorem{rmk}[thm]{Remark}

 \newtheorem{claim}[thm]{Claim}

\newcommand{\N}{\mathcal{N}}

\begin{document}
\title{Homotopy type of the neighborhood complexes of graphs of maximal degree at most $3$ and $4$-regular circulant graphs }
\author{ Samir Shukla\footnote{{Department of Mathematics, Indian Institute of Technology Bombay, India. samirs@math.iitb.ac.in }}}
\maketitle
\begin{abstract}
	To estimate the lower bound for the chromatic number of a graph $G$,  Lov{\'a}sz  associated a simplicial complex 
	$\mathcal{N}(G)$ called the neighborhood complex and relates the topological connectivity of $\mathcal{N}(G)$ to the chromatic number
	of $G$. More generally he proved that
	the  chromatic number of $G$ is bounded below by the topological connectivity of $\mathcal{N}(G)$
	plus $3$.
	
	In this article, we consider the graphs of maximal degree at most $3$ and $4$-regular circulant graphs. We show that each connected component of the neighborhood comple-xes of these graphs  is homotopy equivalent either to a point, to a wedge sum of  circles, to a wedge sum of $2$-spheres $S^2$, to $S^3$, to a garland of $2$-spheres $S^2$ or to a connected sum of tori. 
\end{abstract}

\noindent {\bf Keywords} : Neighborhood complexes, Circulant  graphs, Shellability.


\vspace{.1in}

\hrule

\section{Introduction}

The {\it neighborhood complex}, $\N(G)$ of a graph $G$ is the simplicial complex whose
vertices are all non isolated vertices of $G$ and simplices are those subsets of $V(G)$, which have a common neighbor. 
The concept of neighborhood complex 
was introduced by Lov{\' a}sz \cite{l} in his proof of the Kneser conjecture, that if we split the $n$-subsets of a $(2n+k)$-elements set into 
$k+1$ classes, then one of the classes will contain two  disjoint subsets. To prove this conjecture, Lov{\' a}sz first converted this set theoretic problem into an equivalent problem of the computation of chromatic number of a class of graphs, called Kneser graphs and then relates the connectivity of neighborhood complex to the chromatic number of graph.

A topological space $X$ is said to be $k$-{\it connected} if every map from a $m$-dimensional sphere $S^m \to X$ can be extended to a map from the $(m+1)$-dimensional ball $\mathbb{B}^{m+1} \to X$ form $m = 0, 1, \ldots , k.$
\begin{thm} (Lov{\' a}sz) If $\mathcal{N}(G)$ is $k$-{\it connected}, then  $\chi(G) \geq k+3$.
\end{thm}
Lov{\'a}sz generalize the notion of neighborhood complex to a polyhedral complex Hom$(G$ $,  H)$, called the hom complex,
for graphs $G$ and $H$.  In particular Hom$(K_2, G)$ and $\N(G)$ are homotopy equivalent. The 
$0$-dimensional cells of Hom$(G, H)$ are the 
graph homomorphisms from $G$ to $H$. For more details about hom complexes we refer the reader to 
\cite{BK}. In \cite{BL}, Bj\"{o}rner and Longueville showed that the neighborhood complexes of a family of vertex critical subgraphs of
Kneser graphs -- the stable Kneser graphs, are spheres up to homotopy. In \cite{NS}, Nilakantan and author studied the neighborhood complexes of the exponential graphs $K_{n+1}^{K_n}$. In this  article we compute the homotopy type 
of the neighborhood complexes of $4$-regular circulant graphs.

Let $n \geq 2$ be a positive integer and $S \subset \{ 1, 2, \ldots, n-1\}$. 
The {\it circulant graph} $C_n(S)$ is the graph, whose set of vertices
$V(C_n(S)) =  \{1, 2, \ldots, n\}$  and any two vertices $x$ and $y$
adjacent if and only if $x-y \ (\text{mod \ n}) \in S \cup -S$, where $-S = \{n-a \ | \ a \in S\}$.
Circulant 
graphs are also Cayley graphs of $\mathbb{Z}_n$, the cyclic group on 
$n$ elements. Since $n \notin S$, $C_n(S)$ is a simple graph, {\it i.e.},
does not contains any loop. Further, $C_n(S)$ are $|S \cup -S|$-regular graphs, here $| \cdot |$ denoting the cardinality. It can be easily  verify that $C_n(S)$
is connected
if and only if $S \cup -S$ generates  $\mathbb{Z}_n$.
In this article we restrict ourselves to $|S| = 2$ and for the 
convenience of notation, we write $C_n(s,t)$ in place of $C_n(\{s,t\})$.
Since $C_{n}(s, t) = C_n(n-s, t) = C_n(s, n-t)$, we  assume that 
$s, t \leq \frac{n}{2}$.

\section{Statement of results}
We use the following definition of  garland of topological spaces in  Theorem  \ref{THEOREM3} and Theorem \ref{THEOREM4}.
Let $X_1, X_2, \ldots, X_n$ be topological spaces. A topological space $X =\bigcup\limits_{i=1}^{n} X_i$ is said to be a garland of $X_1, \ldots, X_n$, if $|X_i \cap X_{j} |= 1$ when  $|i-j| = 1 \ (mod \ n)$ and $X_i \cap X_j = \emptyset$ when $|i-j| \neq 1 \ (mod \ n)$   for all $1 \leq i \neq j \leq n$.

\begin{thm}\label{3regular}
	Let $G$ be a connected graph with maximal degree at most $3$ and $G \neq K_4, T$ (see Figure \ref{a} and \ref{b}).
	Each connected component on $\N(G)$ is either contractible or 
	homotopy equivalent to  wedge sum of circles.
\end{thm}

Let $n \geq 5$ and $s, t \in \{1, 2, \ldots, \lfloor \frac{n}{2}\rfloor\}, s \neq t$. To compute the homotopy type of the neighborhood complexes of $4$-regular circulant graphs we divide the set
$I = \{\{s, t\} \ | \ s, t \in \{1, 2, \ldots, \lfloor \frac{n}{2}\rfloor\}, s \neq t\}$ in to the following disjoint
classes: $I_1 = \{\{s, t\}  \in I \ | \ n \in \{2s, 2t, 2(s+t)\}\}, I_2 = \{\{s, t\} \in I \setminus I_1 \ | \  3s= t \ \text{or} \ 3t = s \}, 
I_3 = \{\{s, t\} \in I \setminus (I_1 \cup I_2)\ | \  3s = 5t \ \text{or} \ 3t = 5s \}$ and $I_4 = I \setminus (I_1 \cup I_2 \cup I_3)$.

\begin{thm}\label{THEOREM1}
	Let $\{s, t\} \in I_1$. Each connected component of $\N(C_n(s,t))$ is homotopy equivalent  to 
	
	\begin{itemize}
		\item [(A)] a point or wedge sum of circles, if $2s = n$ or $2t = n$.
		\item[(B)]  a point or $S^1$, if $2s, 2t \neq n$ and $2(s+t) = n$.
	\end{itemize}
\end{thm}

\begin{thm} \label{THEOREM2} Let $\{s,t\} \in I_2$.  Each connected component of $\N(C_n(s,t))$ is homotopy equivalent  to 
	\begin{itemize}
		\item[(A)] $S^3$, if $3s = t $ and $n = 10s \ ( \text{or} \ 3t = s$ and $n = 10t )$.
		\item[(B)] $S^2 \vee S^2$, if $3s = t $ and $n = 12s  \ (\text{or} \ 3t = s \ \text{and} \ n = 12t)$.
		\item[(C)] wedge sum of circles, if $3s = t $ and $n \neq 8s, 10s, 12s \ ( \text{or} \ 3t = s \ \text{and} \ n \neq 8t, 10t, 12t)$.
	\end{itemize}
\end{thm}

\begin{thm}\label{THEOREM3}
	Let $\{s, t\} \in I_3$. Each connected component of $\N(C_n(s, t))$
	is homotopy equivalent to   
	\begin{itemize}
		
		\item[(A)] a garland of  the $2$-dimensional spheres $S^2$, if 
		$ 5s = 3t$ and $n = 4t \ (\text{or} \  5t = 3s \ \text{and} \  n = 4s)$.
		\item[(B)]  $S^2 \vee S^2$, if  $5s = 3t$ and $n = 4s \ (\text{or} \  5t = 3s \ \text{and} \  n = 4t)$.
		
		\item[(C)] wedge sum of circles, if $5s = 3t$ and one of  $ 6s$ or $\frac{14s}{3}$ is equal to $ n \ (\text{or} \  5t = 3s$ and 
		one of 
		$6t$ or $\frac{14t}{3}$ is equal to $ n)$.
		
		\item[(D)] connected sum of tori, if $5s = 3t$ and $4s, 6s, \frac{14s}{3}, 4t \neq n \ (\text{or}\
		5t= 3s$ and $4t,6t,\frac{14t}{3}, $ $4s \neq  n)$.
	\end{itemize}
	
\end{thm}

\begin{thm} \label{THEOREM4}
	Let $\{s, t\} \in I_4$. Each
	connected component of $\N(C_n(s,t))$ is homotopy equivalent  to 
	\begin{itemize}
		\item[(A)]  either $S^1$ or $S^3$, if one of the $3s-t, 3t-s, 3s+t \ \text{or}\ 3t+s$ is equal to $ n$.

		\item[(B)] a garland of  the  $2$-dimensional spheres $S^2$, if 
		$ 3s-t, 3t-s, 3s+t, 3t+s \neq n$ and, $4t=n$ or $4s = n$.
		\item[(C)] connected sum of tori, if $3s-t, 3t-s, 3s+t, 3t+s , 4s, 4t \neq n$.
		
	\end{itemize}
	
\end{thm}

The following theorem can be considered as a special case of Theorem \ref{THEOREM4} (C).
\begin{thm} \label{special}
	Let $n = pq$, where gcd$(p, q) = 1$. Let $s, t \in \{1, \ldots, n-1\}$ such that
	$2s, 2t, 2(s+t), 3s-t, 3t-s, 3s+t, 3t+s, 4s, 4t \not\equiv 0 \ (\text{mod} \ n)$.
	If $s = \frac{p-q}{2}$ and  $t = \frac{p+q}{2}$ or, $s = \frac{p^2-q}{2} $ and 
	$t = \frac{p^2+q}{2}$,  then $\N(C_{n}(s, t))$ is homotopy equivalent  to  torus. 
	
\end{thm}

\section{Preliminaries}

\subsection{Graph}
A  graph $G$ is a  pair $(V(G), E(G))$,  where $V(G)$  is the set of vertices
of $G$  and $E(G) \subset V(G)\times V(G)$
denotes the set of edges.
If $(x, y) \in E(G)$, it is also denoted by $x \sim y$. 
A {\it subgraph} $H$ of $G$ is a graph with $V(H) \subset V(G)$ and $E(H) \subset E(G)$.
For a subset $U \subset V(G)$, the induced subgraph $G[U]$ is the subgraph whose set of vertices  $V(G[U]) = U$
and the set of edges
$E(G[U]) = \{(a, b) \in E(G) \ | \ a, b \in U\}$.

A {\it graph homomorphism} from  $G$ to $H$ is a function
$\phi: V(G) \to V(H)$ such that, $(v,w) \in E(G) \implies (\phi(v),\phi(w)) \in E(H).$ A graph homomorphism $f$ is called an {\it isomorphism} if $f$ is bijective and $f^{-1}$ is also a graph homomorphism. Two graphs are called {\it isomorphic},  if there exists an isomorphism between them. If $G$ and $H$ are isomorphic, we write $G \cong H$.
The {\it chromatic number} $\chi(G)$ of a graph $G$ is defined as
$\chi(G) := \text{min}\{n \ | \ \exists \ \text{a graph} $ $ \text{ homomorphism from } G \ \text{to} \ K_n\} $.
Here, $K_n$ denotes a complete graph on $n$ vertices.

Let $G$ be a graph and $v$ be a vertex of $G$. The {\it
	neighbourhood  of $v$ }is defined as $N(v)=\{ w \in V(G) \ |  \
(v,w) \in E(G)\}$. If $A\subset V(G)$, the set of neighbours of $A$
is defined as $N(A)= \{x \in  V(G) \ | \ (x,a) \in E(G)
\,\,\forall\,\, a \in A \}$. The {\it degree} of a vertex $v$ is $|N(v)|$. A graph is said to be {\it $d$-regular}, 
if each vertex has degree $d$. The {\it maximal degree} of $G$ is the maximum of  the degree of vertices of $G$. 

\subsection{Simplicial complex}
A {\it finite abstract simplicial complex X} is a collection of
finite sets such that if $\tau \in X$ and $\sigma \subset \tau$,
then $\sigma \in X$. The elements  of $X$ are called {\it simplices}
of $X$. The  dimension of a simplex $\sigma$ is equal to $|\sigma| - 1$. 
The dimension of an abstract  simplicial complex is the maximum of the dimensions of its simplices. The $0$-dimensional
simplices are called vertices of $X$. If $\sigma \subset \tau$, we say that
$\sigma$ is a face of $\tau$.  If a simplex has dimension $k$, it is said to
be $k${\it -dimensional} or   $k$-{\it simplex}.  The {\it boundary} of a $k$-simplex
$\sigma $ is the simplicial complex, consisting of  all  faces of $\sigma$
of dimension $\leq k-1$ and it is denoted by $Bd(\sigma).$
A simplex which is not a face of any other
simplex is called a  {\it maximal simplex}. The set of maximal simplices of $X$ is denoted by $M(X)$. A simplicial complex is called
{\it pure $d$-dimensional}, if all of its maximal simplices are of dimension $d$.

Let  $X$ be a simplicial complex and $\tau, \sigma \in X$ such that
$\sigma \subsetneq \tau$ and  $\tau$ is the only maximal simplex in $X$ that contains $\sigma$.
A  {\it simplicial collapse} of $X$ is the simplicial complex $Y$ obtained from $X$ by
removing all those simplices $\gamma$  of $X$ such that
$\sigma \subseteq \gamma \subseteq \tau$. Here, $\sigma$ is called a {\it free face} of
$\tau$ and $(\sigma, \tau)$ is called a {\it collapsible pair}. We denote this collapse
by $X \searrow Y$. In particular, if 
$X \searrow Y$, then $X \simeq Y$.

\subsection{Shellability}

A pure  $d$-dimensional simplicial complex $X$ is said to be {\it shellable},
if its maximal simplices can be ordered $\Gamma_1, \Gamma_2 \ldots, \Gamma_t$ in such a way
that the subcomplex $(\bigcup\limits_{i = 1}^{k-1} \Gamma_i) \cap \Gamma_k$ is pure
and $(d-1)$-dimensional for all $k = 2, \ldots, t$. Here, this ordering $\Gamma_1, \ldots, \Gamma_t$
of maximal simplices is called a {\it shelling order}.
A maximal simplex $\Gamma_k$ is called spanning with respect to the given shelling
order if $Bd(\Gamma_k) \subseteq \bigcup\limits_{i = 1}^{k-1} \Gamma_i$.

From \cite[Theorem $12.3$]{BK} we have the following result which tells us about the homotopy type of shellable complexes.

\begin{prop} \label{shellable}
	Assume that $X$ be a shellable  simplicial complex,
	with $\Gamma_1, \Gamma_2, \ldots, \Gamma_t$ being the corresponding shelling order of the maximal simplices,
	and $\sum$ being the set of spanning simplices. Then 
	\begin{center}
		$X \simeq \bigvee\limits_{\sigma \in \sum} S^{dim \ \sigma}$
	\end{center}
\end{prop}

\begin{rmk} \label{remarkshe}
	Any connected $1$-dimensional simplicial complex is always shellable and therefore it is either contractible or homotopy equivalent to wedge sum of circles.
\end{rmk}

\subsection{Folding}
Let $G$ be a graph and $N(u) \subset N(v)$  for $u,v \in V(G), u \neq v$. In this case,  the graph homomorphism  $V(G) \to V(G) \setminus \{u\}$, which sends $u$ to $v$ and fixes all other vertices, called {\it folding} and 
the graph $G\setminus \{u\}$ is called a {\it fold} of $G$.  Here,  $V(G\setminus
\{u\}) = V(G) \setminus \{u\}$ and the edges in the subgraph
$G\setminus \{u\}$ are all those edges of $G$ which do not contain
$u$. 

\begin{prop}(\cite{BK}, Proposition $4.2$ and Proposition $5.1$) \label{fold} \\
	Let $G$ be a graph and $u \in V(G)$. If $G$ is folded on to $G \setminus \{u\}$, then $\N(G)$ is of same homotopy type as $\N(G \setminus \{u\})$.
\end{prop}

\section{Proofs}
In this article $[n]$ denotes the set $\{1,2,\ldots, n\}$. Throughout this paper, all the graphs are finite and simple,
{\it i.e.}, not contain loops.

\begin{proof}[Proof of Theorem \ref{3regular}]
	
	Since, by Proposition \ref{fold}, folding preserve the homotopy type in neighborhood complex, without loss of generality we assume that $G$ cannot be folded onto any of its subgraphs. If the maximal degree of $G$ is $1$, then $G \cong K_2$. If the maximal degree of $G$ is $2$, then $\N(G)$ is a $1$-dimensional complex and the result follows from Remark \ref{remarkshe}. So, assume that maximal degree of $G$ is $3$. In this case, $\N(G)$ is  $2$-dimensional.
	We show that $\N(G)$ collapses  to a $1$-dimensional subcomplex. 
	The $2$-dimensional simplices of $\N(G)$ are the neighborhoods of vertices of degree $3$. Let $\sigma = \{\alpha, \beta, \gamma\}$ be a $2$-simplex  of $\N(G)$. Then, there exists $x\in V(G)$ such that 
	$N(x) = \{\alpha, \beta, \gamma\}$. We consider the following two cases.
	
	\vspace{0.15 cm}
	
	\noindent{\bf{Case 1.}} At least one of the $(\alpha, \beta), (\beta, \gamma)$ or $(\alpha, \gamma) \in E(G)$.
	
	Without loss of generality assume that $(\alpha, \beta) \in E(G)$ (see Figure \ref{c}).
	If $\alpha$ and $\beta$ do not have any  common neighbor other than $x$, then $(\{\alpha, \beta\}, \sigma )$ 
	is a collapsible pair and therefore $\N(G) \searrow \N(G) \setminus \{\sigma, \{\alpha, \beta\}\}$.
	
	Assume $\alpha, \beta$ have a common neighbor $y \neq x$. Since, $G \neq K_4, y \neq \gamma$. If $y \sim \gamma$, then $N(y) = \{\alpha, \beta, \gamma\}$
	and $(\{\alpha, \beta\}, \sigma)$ is a  collapsible pair. If $y \nsim \gamma$,
	then since $N(\alpha) = \{x, \beta , y\}, N(\beta) = \{x, \alpha, y\}$ and maximal degree of $G$ is $3$, we see that 
	$(\{\beta, \gamma\}, \sigma)$ is a collapsible pair.

	\begin{figure} [H]
		\begin{subfigure}[]{0.3\textwidth}
			\centering
			\begin{tikzpicture}
			[scale=0.4, vertices/.style={draw, fill=black, circle, inner
				sep=0.5pt}]
			\node[vertices, label=above:{}] (1) at (0,0) {};
			\node[vertices, label=below:{}] (2) at (0,4) {};
			\node[vertices,label=above:{}] (3) at (4,0) {};
			\node[vertices,label=left:{}] (4) at (4,4) {};
			
			\foreach \to/\from in
			{1/2,1/3,1/4, 2/3, 2/4, 4/3} \draw [-] (\to)--(\from);
			
			\end{tikzpicture}
			\caption{$K_4$} \label{a}
		\end{subfigure}
		\begin{subfigure}[]{0.3\textwidth}
			\centering
			\begin{tikzpicture}
			[scale=0.4, vertices/.style={draw, fill=black, circle, inner
				sep=0.5pt}]
			\node[vertices, label=above:{}] (1) at (0,4.5) {};
			\node[vertices, label=below:{}] (2) at (0,1.5) {};
			\node[vertices,label=above:{}] (3) at (3,4.5) {};
			\node[vertices,label=left:{}] (4) at (1.5,3) {};
			\node[vertices,label=right:{}] (5) at (4.5,3) {};
			\node[vertices,label=right:{}] (6) at (4.5,0) {};
			\node[vertices,label=below:{}] (7) at (1.5,0) {};
			\node[vertices,label=right:{}] (8) at (3,1.5) {};
			\foreach \to/\from in
			{1/2,1/3,1/4, 2/7, 2/8, 4/7, 4/5,3/5, 3/8, 5/6, 6/7, 6/8} \draw [-] (\to)--(\from);
			
			\end{tikzpicture}
			\caption{$T$} \label{b}
		\end{subfigure}
		\begin{subfigure}[]{0.3\textwidth}
			\centering
			\begin{tikzpicture}
			[scale=0.4, vertices/.style={draw, fill=black, circle, inner
				sep=0.5pt}]
			\node[vertices, label=left:{$x$}] (a) at (0,2.5) {};
			\node[vertices, label=right:{$\alpha$}] (b) at (2,0) {};
			\node[vertices,label=right:{$\beta$}] (c) at (2,2.5) {};
			\node[vertices,label=right:{$\gamma$}] (d) at (2,5) {};
			\foreach \to/\from in
			{a/b,a/c,a/d, b/c} \draw [-] (\to)--(\from);
			\end{tikzpicture}
			\caption{} \label{c}
		\end{subfigure}
		\vspace{0.8 cm}
		
		\begin{subfigure}[]{0.4\textwidth}
			\centering
			\begin{tikzpicture}
			[scale=0.4, vertices/.style={draw, fill=black, circle, inner
				sep=0.5pt}]
			\node[vertices, label=left:{$x$}] (a) at (0,2) {};
			\node[vertices, label=below:{$\alpha$}] (b) at (2,0) {};
			\node[vertices,label=below:{$\beta$}] (c) at (2,2) {};
			\node[vertices,label=below:{$\gamma$}] (d) at (2,4) {};
			\node[vertices,label=right:{$y$}] (e) at (4,1) {};
			\node[vertices,label=right:{$z$}] (f) at (6,1) {};
			\foreach \to/\from in
			{a/b,a/c,a/d, b/e,c/e} \draw [-] (\to)--(\from);
			
			\path
			
			(f) edge[bend left= 30] (b) 
			(f) edge[bend right= 30] (c) ;
			
			\end{tikzpicture}
			\caption{} \label{d}
		\end{subfigure}
		\begin{subfigure}[]{0.5\textwidth}
			\centering
			\begin{tikzpicture}
			[scale=0.4, vertices/.style={draw, fill=black, circle, inner
				sep=0.5pt}]
			\node[vertices, label=left:{$x$}] (a) at (4,2) {};
			\node[vertices, label=below:{$\alpha$}] (b) at (6,0) {};
			\node[vertices,label=below:{$\beta$}] (c) at (6,2) {};
			\node[vertices,label=below:{$\gamma$}] (d) at (6,4) {};
			\node[vertices,label=below:{$y$}] (e) at (8,1) {};
			\node[vertices,label=above:{$z$}] (f) at (8,3) {};
			\node[vertices,label=below:{$t$}] (g) at (2,2) {};
			\node[vertices,label=right:{$a$}] (h) at (10,1) {};
			\node[vertices,label=right:{$b$}] (i) at (10,3) {};
			\node[vertices,label=below:{$c$}] (j) at (0,2) {};
			\foreach \to/\from in
			{a/b,a/c,a/d, b/e, c/e, f/d, c/f, e/h, f/i, g/j} \draw [-] (\to)--(\from);
			
			\path
			
			(g) edge[bend left= 25] (d) 
			(g) edge[bend right= 25] (b) ;
			\end{tikzpicture}
			\caption{} \label{e}
		\end{subfigure}
		
		\caption{} \label{ex}
	\end{figure}

	\noindent{\textbf{Case 2.}}
	None of the $(\alpha, \beta), (\beta, \gamma)$ or $(\alpha, \gamma) \in E(G)$.
	
	Since maximal degree of $G$ is $3$, any $1$-dimensional simplex of $\N(G)$
	can be a face of at most three $2$-simplices.
	If there exists a face, say $\{\alpha, \beta\}$ of $\{\alpha, \beta, \gamma\}$ which
	is a face of three $2$-simplices of 
	$\N(G)$ (see Figure \ref{d}), then since $x$ is not contained in any $3$-cycle (a graph on $3$ vertices, where each vertex has degree $2$), $\gamma \neq y, z$. In this case, 
	$(\{ \alpha, \gamma \}, \sigma)$ is a collapsible pair.
	
	Assume none of the  $1$-dimensional face of $\{\alpha, \beta, \gamma\}$ is contained in three maximal simplices, {\it i.e.},
	each face is contained in at most two $2$-simplices of $\N(G)$. If some face of $\sigma$ is contained in only one $2$-simplex, then clearly that face will be a free  face of $\sigma$. So, assume that each face of $\sigma$ contained in exactly two $2$-simplices of $N(G)$  (see Figure \ref{e}).
	If $y = z= t$, then since maximal degree of $G$ is $3$, at least  one of the 1-dimensional face of
	$\{\alpha, \beta , \gamma\}$ will be a free face.
	
	Suppose exactly two elements of $\{y, z, t\}$ are same, say $y = z$. In this case $\{\beta, \gamma\}$ is a free face of
	$\{\alpha, \beta , \gamma\}$. Assume $| \{y, z, t\}| = 3$.  Since, $G$ cannot be folded onto any of its subgraph, degree of $y, z$ and $t$ must be $3$. Let $a, b$ and $c$ are as depicted in Figure \ref{e} and
	$\{a,b,c\} \cap \{\alpha, \beta, \gamma\} = \emptyset$.
	
	If $|\{a, b, c\}| = 3$, then  since the common neighbor of $\gamma$ and $c$ is only $t$, we see that 
	$(\{\gamma, c\}, N(t))$ is a collapsible pair. Hence,
	$\N(G) \searrow \N(G) \setminus \{N(t), \{\gamma, c\}\}$. Now
	$(\{\alpha, \gamma\},$ $ N(x))$ is a collapsible pair in  $\N(G) \setminus \{N(t), \{\gamma, c\}\}$
	and therefore $\N(G) \searrow \N(G) \setminus \{N(x),$ $ N(t),$ $  \{\alpha, \gamma\} , \{\gamma, c\}\}.$

	Now, let $|\{a, b, c\}| = 2$.
	Without loss of generality we can assume that $b= c$. Then $(\{\beta, a\}, N(y))$
	is a collapsible pair and therefore $\N(G) \searrow \N(G) \setminus \{N(y), \{\beta, a\}\}$.
	Now  $(\{\alpha, \beta\}, N(x))$ is a collapsible pair in  $\N(G) \setminus \{N(y), \{\beta, a\}\}$
	and therefore $\N(G) \searrow \N(G) \setminus \{N(x), N(y), $ $ \{\alpha, \beta\} , \{\beta,a\}\}.$ 
	If $a = b= c$, then  $G \cong T$, which is not possible.

	Thus $\N(G)$ collapses to a $1$-dimensional subcomplex.
	From Remark \ref{remarkshe}, each connected component of $\N(G)$ is either contractible or 
	homotopy equivalent to wedge sum of circles.  
	
\end{proof}

\begin{rmk}
	It has been shown in \cite{l} that the neighborhood complex of any non-bipartite graph is connected and is never homotopically trivial. Therefore,
	the nieghborhood complex of any non-bipartite graph of maximal degree at most $3$ is homotopy equivalent to a wedge sum of circles.  
\end{rmk}

We now fix some notations. Throughout this article, if we write an integer $r$ as a vertex of $C_n(s,t)$,
it is understood that we are taking $r$ modulo $n$. Further,
for any two vertices $v_1$ and $ v_2$ of $C_n(s, t)$, $v_1 = v_2$ means  $v_1 \equiv v_2 \ (\text {mod} \ n)$. Since $n$ is fixed, for any two integers
$x$ and $y$ such that $x \equiv y \ (\text{mod} \ n)$, if no confusion arises,
we just write $x \equiv y$.  

For any set $X \subset \mathbb{Z}$ and 
any integer $r$, let  $X+r = \{x+r \ | \ x \in X \}$. Observe that, for any $k \in [n]= V(C_n(s,t))$,
the neighborhood of $k$,  $N(k) = \{s+k, t+k, n-s+k, n-t+k\}$. Since $n \geq 5$,
if $s, t \neq \frac{n}{2}$, then $C_n(s,t)$ is a $4$-regular graph, {\it i.e.}, $|\{s+k, t+k, n-s+k, n-t+k\}| = 4$.
Since $N(k+r) = N(k)+r$ for any $k, r \in [n]$,
observe that any two  connected components of $\N(C_n(s,t))$
are homeomorphic. 

\vspace{0.3 cm}

\begin{proof}[Proof of Theorem \ref{THEOREM1}]
	If  $2s = n $ or $2t = n$, then $C_n(s,t)$ will be a $3$-regular graph and the result follows from Theorem \ref{3regular}.
	Assume $2s, 2t \neq n$ and $2(s+t) = n$. We consider the following two cases.
	
	\vspace{0.15 cm}
	
	\noindent{\textbf{Case 1.}}  $|t-s| = \frac{n}{4}$.
	
	Without loss of generality we can assume that $t > s$ and let $t-s = m$.
	Then  $4m = n, s+t = 2m, 2s = m$. Let  $k \in [n].$ Now,
	$N(k+m) = \{s+k+m, t+k+m, n-s+k+m, n-t+k +m\}$ and 
	$N (k) = \{s+k, t+k, n-s+k, n-t+k\}$. Here, 
	$ s+k + m =  k+t$, 
	$t + k+ m  = n - t + k$, 
	$n - s + k +m \equiv m-s+k  = s+k$ and 
	$n - t + k + m = n-s+k$. Hence $N(k) = N(k+m) = N(k+2m) = N(k+3m)$ for all $k \in [n]$.
	For any $i, j \in [m], i \neq j$,  it can be easily check that $N(i) \cap N(j) = \emptyset$. 
	
	Thus 
	$\mathcal{N}(C_n(s,t))$ is consists of $m$ disjoint  simplices of dimension $3$ and therefore it is homotopy equivalent to $m$ distinct points.
	
	\vspace{0.15 cm}
	
	\noindent{\textbf{Case 2.}} $|t-s| \neq \frac{n}{4}$. 
	
	Let $t+s = p.$ Then $n = 2p$.
	Since $n-t+k = n-t-s+s+k = p+s+k$,
	$n-s+k = p+t+k$,
	we see that 
	$N(k) = N(k+p)$ for all $k \in [n]$.  
	Since $N(s+k) = \{2s+k, s+t+k, k, 2s+p+k\}$ and $N(t+k) = \{s+t+k, 2t+k,  2t+p+k, k\}$,
	$N(s+k) \cap [p] = N(t+k) \cap [p]$ implies that  $\{ 2s+k, 2s+p+k \} \cap  [p] = \{2t+k, 2t+p+k\} \cap [p]$. Since 
	$s \neq t, 2s+k \not\equiv 2t+k$. However,  $2s+k \equiv 2t+p+k$ implies that
	$ 3t \equiv s$. Since $t, s < \frac{n}{2}$, $3t = s$ or $3t = s+n$. If $3t = s$, then $ p= 4t$.
	But, then $s-t = 2t = \frac{p}{2} = \frac{n}{4}$, which is a contradiction. If $3t = s+n = s+2t+2s$,
	then $3s = t$ and $p = 4s$. Here,  $t - s = 2s =  \frac{p}{2} = \frac{n}{4}$, which is not possible. By an argument similar as above, 
	$2s+p+k \notin \{2t+k, 2t+p+k\}$. Hence $N(s+k) \cap [p] \neq N(t+k) \cap [p] \ \forall \ k \in [n]$.

	Since $N(s+k) = N(p+s+k) = N(n-t+k), N(t+k) = N(p+t+k) = N(n-s+k)$, we conclude that 
	each vertex $k$ belongs to exactly two maximal simplices, namely
	$N(s+k)$ and $N(t+k)$. Further, $N(i) \cap [p] \neq N(j) \cap [p] \ \forall \ 1 \leq i \neq j \leq p$.

	Observe that,  $|N(k) \cap \{1,\ldots, p\} |= 2$ for all $k \in [n]$.
	Since $N(k) = N(k+p)$ for all $k \in [n]$,  $C_n(s,t)$ folded onto the induced
	subgraph $C_n(s,t)[\{1,\ldots, p\}]$. Hence 
	$\mathcal{N}(C_n(s,t)) \simeq \mathcal{N}(C_n(s,t)[\{1, \ldots, p\}])$,
	by Proposition \ref{fold}. Since each vertex $k$ of $C_n(s,t)$
	belongs to exactly two maximal simplices $N(s+k) = N(p+s+k)$ and $N(t+k) = N(p+t+k)$ of $\N(C_n(s,t))$, and 
	$N(i) \cap [p] \neq N(j) \cap [p] \  1 \leq i \neq j \leq p$, we see that  
	each vertex  $x \in [p]$ also belongs to exactly two maximal simplices of
	$\mathcal{N}(C_n(s,t)[\{1, \ldots, p\}])$. Further, since $\mathcal{N}(C_n(s,t)[\{1, \ldots, p\}])$
	is a $1$-dimensional complex, the connected components of $\mathcal{N}(C_n(s,t)[\{1, \ldots, p\}])$ cannot be contractible. The 
	result follows from Remark \ref{remarkshe}.
	
\end{proof}

%
\begin{proof}[Proof of Theorem \ref{THEOREM2}(A)]
	Let $3s = t$ and $n = 10s$. For each $1 \leq i \leq s$, let $G_i$ be the subgraph of $C_n(s,t)$ induced by the vertex set $\{i, i+s, i+2s, i+3s, i+4s, i+5s, i+6s, i+7s, i+8s, i+9s\}$. Observe that each $G_i$ is isomorphic to $C_{10}(1,3)$ and $C_n(s, t) \cong \bigsqcup\limits_{i=1}^{s}G_i$. Therefore $\N(C_n(s, t)) \cong \bigsqcup\limits_{i=1}^{s} \N(C_{10}(1,3))$. It can be easily verified that $\N(C_{10}(1,3))$ is homeomorphic to disjoint union of two copies of simplicial boundary of a $4$-simplex, namely the subcomplex  $N(1) \cup N(3) \cup N(5) \cup N(7) \cup N(9) = Bd(\{2, 4, 6,8,10\})$  and the subcomplex $N(2) \cup N(4) \cup N(6) \cup N(8) \cup N(10) = Bd(\{1,3,5,7,9\})$.
	
	We conclude that 
	$\N(C_n(s, t)) \cong \bigsqcup\limits_{2s-copies} S^3$. The case $3t=s$ and $n = 10t$, follows by symmetry. 
	
\end{proof}

%
%
%
%
%

\begin{proof}[Proof of Theorem \ref{THEOREM2}(B)]
	Let $3s = t$ and $n = 12s$. 
	For each $1 \leq i \leq s$, let $G_i$ be the subgraph of $C_n(s,t)$ induced by the vertex set $\{i, i+s, i+2s, i+3s, i+4s, i+5s, i+6s, i+7s, i+8s, i+9s, i+10s, i+11s\}$. Observe that each $G_i$ is isomorphic to $C_{12}(1,3)$ and $C_n(s, t) \cong \bigsqcup\limits_{i=1}^{s}G_i$. We now compute $\N(C_{12}(1,3))$.

	For each $i \in \{1 , 2\}$, let $A_i$ be the subcomplex of $\N(C_{12}(1,3))$, where the set of maximal simplices $M(A_i) = \{N(i), N(i+2), N(i+4), N(i+6), N(i+8), N(i+10)\}$. Then $\N(C_{12}(1,3))$ = $A_1 \sqcup A_2$. It can be easily checked that for each $k$, $ (\{k+1, k+3, k+9\}, N(k))$ is a collapsible pair in $\N(C_{12}(1,3))$
	and therefore $N(k) \searrow  \tau_k^1 = \{k+1, k+3, k+11\},
	\tau_k^2 = \{k+3, k+9, k+11\}$ and $\tau_k^3 = \{k+1, k+9, k+11\}$. Since,  $\tau_{k+10}^1 = \{k+11, k+1, k+9\} = \tau_{k}^3$ for all $k \in [12]$, we see that 
	$\N(C_{12}(1,3))$ collapses to a subcomplex $\Delta$, where the  set of maximal simplices
	$M(\Delta) = \{\tau_k^i \ | \ i \in \{1,2\}, k \in [12]\}$.  For, $i \in \{1,2\}$, let $\Delta_i= \Delta \cap A_i$. Then, $\Delta = \Delta_1 \sqcup \Delta_2$. It can be easily verified that each   $\Delta_i$ is a shellable complex and the shelling order is given by 
	\begin{eqnarray} \label{shellingorder}
	\tau_i^1, \tau_{i}^2, \tau_{i+2}^1, \tau_{i+2}^2, \tau_{i+4}^1, \tau_{i+4}^2, \tau_{i+6}^2, \tau_{i+6}^1,
	\tau_{i+8}^1, \tau_{i+8}^2, \tau_{i+10}^1, \tau_{i+10}^2.
	\end{eqnarray}
	
	Here,  $\tau_{i+8}^2$
	and $\tau_{i+10}^2$ are spanning simplices with respect to the shelling order \ref{shellingorder}.  Hence, $\Delta_i \simeq S^2 \vee S^2$ and by Proposition \ref{shellable}.  We conclude that $\N(C_n(s, t)) \cong \bigsqcup\limits_{2s-copies} S^2\vee S^2$.

	If $3t = s$ and $n = 12t$, the result follows by symmetry. 
\end{proof}
\begin{proof}[Proof of Theorem \ref{THEOREM2}(C)]
	Let $3s = t$ and $8s, 10s, 12s \neq n$. Observe that, $2s, 4s, 6s, 8s, $ $12s \not\equiv 0$. 
	For each $k \in [n], N(k) = \{s+k, 3s+k, n-s+k, n-3s+k\}$.
	Since $4s, 6s, 8s, 10s, $ $ 12s \not\equiv 0 $, we see that 
	$N(3s+k) \cap N(n-3s+k) = \{k\}$ and therefore $(\{3s+k, n-3s+k\}, N(k))$
	is a collapsible pair. Hence,
	$N(k) \searrow \tau_k^1 = \{3s+k, s+k, n-s+k\}$ and $\tau_k^2 = \{n-3s+k, s+k, n-s+k\}$ $\forall \ k \in [n]$. Since 
	$\tau_{n-2s+k}^1 = \{s+k, n-s+k, n-3s+k\} = \tau_k^2$, $\N(C_{n}(s,t))$
	collapses to a subcomplex $\Delta$, where 
	$M(\Delta) = \{\tau_k^1 \ |  \ k \in [n]\}$. 
	
	Using the fact that $2s, 4s, 6s, 8s, 12s \not\equiv 0$, it can be easily check that 
	$N(3s+k) \cap N(n-s+k) = \{2s+k, k\}$. Since $\{3s+k, n-s+k\} \nsubseteq \tau_{2s+k}^1 = \{5s+k, 3s+k,s+k \}$,
	we conclude that
	$(\{3s+k, n-s+k\}, \tau_k^1)$ is a collapsible pair.
	Hence, for each $k \in [n]$, 
	$\tau_k^1 \searrow \delta_k^1 = \{3s+k, s+k\}$ and
	$\delta_k^{2} = \{n-s+k, s+k\}$. Since 
	$\delta_{n-2s+k}^1 = \{s+k, n-s+k\} = \delta_k^2$. Applying the collapsible pairs 
	of the type $(\{3s+k, n-s+k\}, \tau_k^1)$ for each $k \in [n]$, $\Delta$ collapses to a 
	$1$-dimensional subcomplex $\Delta'$, 
	where $M(\Delta') = \{\delta_k^1  \ | \  k \in [n]\}$.
	
	Since each vertex $k \in [n]$ belongs to $\delta_{n-s+k}^1$ and $\delta_{n-3s+k}^1$,
	and $\delta_{n-s+k}^1 \neq \delta_{n-3s+k}^1$,  connected components of $\Delta'$ cannot be contractible. Therefore,
	by Proposition \ref{shellable}, each connected component of $\Delta'$ is homotopy equivalent to wedge sum of circles.
\end{proof}

\begin{proof}[Proof of Theorem \ref{THEOREM3}(A)]
	Let $5s = 3t$ and $n = 4t$. Since $s \neq t, n \neq 4s$. Further, since $3s+t = n$ implies that $s = t$ and  $3s-t = n$ implies that
	$3s = 5t = \frac{25}{3}s$, we see that $3s+t, 3s-t \neq n $. The result follows from  Theorem \ref{gerland}.
	The case $5t = 3s$ and $n = 4s$, follows by symmetry. 
	
\end{proof} 

\begin{proof}[Proof of Theorem \ref{THEOREM3}(B)]
	Let $5s = 3t$ and $n = 4s$. Since $5$ and $3$ are relatively prime, there exists $s'$ such that $s = 3s'$ and $t = 5s'$.
	Here, $n = 12s'$ and 
	$N(k) = \{3s'+k, 5s'+k, 9s'+k, 7s'+k\} \ \forall \ k \in [n]$. For  $i \in \{1 , \ldots, 2s'\}$, let $B_i$ be the subcomplex of $\N(C_n(s,t))$, where 
	$M(B_i) = \{N(i +2s'l) \ | \ l \in \{0,1,2,3,4,5\} \}$. It can be easily checked that the set of vertices of $B_i$ is given by 
	$V(B_i) = \{ i+s', i+3s', i+5s', i+7s', i+9s', i+11s'\}$ and
	$V(B_l) \cap V(B_m) = \emptyset \ \forall \ 1 \leq l \neq m \leq 2s'.$ 
	Hence $B_l \cap B_m = \emptyset \ \forall \ 1 \leq l \neq m \leq 2s'$.
	
	Let $k \in [n]$. Since $N(3s'+k) \cap N(5s'+k) \cap N(9s'+k) = \{k\}$, 
	$(\{3s'+k, 5s'+k, 9s'+k\}, N(k))$ is a collapsible pair. Therefore 
	$N(k) \searrow \tau_k^1 = \{3s'+k, 9s'+k, 7s'+k\},
	\tau_{k}^2 = \{5s'+k, 9s'+k, 7s'+k\}$ and $\tau_k^3 = \{3s'+k, 5s'+k, 7s'+k\}$. Since,  $\tau_{n-2s'+k}^2 = 
	\{3s'+k, 7s'+k, 5s'+k\} = \tau_{k}^3$, we conclude that  $\N(C_n(s,t))$ 
	collapses to a subcomplex $\Delta$, where
	$M(\Delta) = \{\tau_{k}^i \ | \ i \in \{1,2\}, k \in [n]\}$.

	For $i \in \{1, \ldots, 2s'\}$, let $\Delta_i = \Delta \cap B_i$. Since
	$B_i \cap B_j = \emptyset, \Delta_i \cap  \Delta_j = \emptyset \ \forall \ 1 \leq i \neq j \leq 2s'$. Further,
	$M(\Delta_i) = \{\tau_{i+2s'l}^j \ | \ j \in \{1,2\},
	l \in \{0,1,2,3,4,5\}\}$ $\forall \ 1 \leq i \leq 2s'$.
	
	We now show that each $\Delta_i$ is a shellable complex and the shelling order is given by 
	\begin{eqnarray} \label{shellingorder2}
	\tau_i^1, \tau_{i}^2, \tau_{i+2s'}^1, \tau_{i+2s'}^2, \tau_{i+4s'}^1, \tau_{i+4s'}^2, \tau_{i+6s'}^1, \tau_{i+6s'}^2,
	\tau_{i+8s'}^1, \tau_{i+8s'}^2, \tau_{i+10s'}^1, \tau_{i+10s'}^2. 
	\end{eqnarray}
	For each $i \in [n]$ and $0 \leq l \leq 5$, let $\tau_{i+2ls'} = \tau_i^1 \cup \tau_i^2 \cup \ldots \cup \tau_{i+2ls'}^1 \cup \tau_{i+2ls'}^2$. The following are easy to verify.
	
	$\tau_i^1 = \{3s'+i, 9s'+i,7s'+i \}, \tau_i^2 = \{5s'+i, 9s'+i, 7s'+i\}$ and 
	$M(\tau_i^1 \cap \tau_i^2) = \{\{9s'+i, 7s'+i\}\}$.
	
	$\tau_{i+2s'}^1 = \{5s'+i, 11s'+i, 9s'+i\}$ and $M(\tau_i \cap \tau_{i+2s'}^1) = \{\{5s'+i, 9s'+i\}\}$.
	
	$\tau_{i+2s'}^2 = \{7s'+i, 11s'+i, 9s'+i\}$ and $M((\tau_{i} \cup \tau_{i+2s'}^1) \cap \tau_{i+2s'}^2)
	= \{\{11s'+i, 9s'+i\}, \{9s'+i, 7s'+i\}\}$.
	
	$\tau_{i+4s'}^1 = \{7s'+i, s'+i, 11s'+i\}$ and $M( \tau_{i+2s'} \cap \tau_{i+4s'}^1)
	= \{\{7s'+i, 11s'+i\}\}$.
	
	$\tau_{i+4s'}^2 = \{9s'+i, s'+i, 11s'+i\}$ and $M(( \tau_{i+2s'} \cup \tau_{i+4s'}^1) \cap 
	\tau_{i+4s'}^2) = \{\{s'+i, 11s'+i\}, \{9s'+i, 11s'+i\}\}$.
	
	$\tau_{i+6s'}^1 = \{9s'+i, 3s'+i, s'+i\}$ and $M(\tau_{i+4s'} \cap 
	\tau_{i+6s'}^1) = \{\{3s'+i, 9s'+i\}, \{s'+i, 9s'+i\}\}$.
	
	$\tau_{i+6s'}^2 = \{11s'+i, 3s'+i, s'+i\}$ and $M(( \tau_{i+4s'} \cup \tau_{i+6s'}^1) \cap 
	\tau_{i+6s'}^2) = \{\{11s'+i, s'+i\}, \{3s'+i, s'+i\}\}$.
	
	$\tau_{i+8s'}^1 = \{11s'+i, 5s'+i, 3s'+i\}$ and $M( \tau_{i+6s'} \cap 
	\tau_{i+8s'}^1) = \{\{11s'+i, 3s'+i\}, \{11s'+i, 5s'+i\}\}$.

	$\tau_{i+8s'}^2 = \{s'+i, 5s'+i, 3s'+i\}$ and $M(( \tau_{i+6s'} 
	\cup \tau_{i+8s'}^1) \cap 
	\tau_{i+8s'}^2) = \{\{s'+i, 3s'+i\}, \{5s'+i, 3s'+i\}\}$.
	
	$\tau_{i+10s'}^1 = \{s'+i, 7s'+i, 5s'+i\}$ and $M( \tau_{i+8s'} \cap 
	\tau_{i+10s'}^1) = Bd(\tau_{i+10s'}^1)$.
	
	$\tau_{i+10s'}^2 = \{3s'+i,7s'+i,5s'+i\}$ and $M(( \tau_{i+8s'} \cup \tau_{i+10s'}^1) \cap 
	\tau_{i+10s'}^2) = Bd(\tau_{i+10s'}^2)$.

	Thus the order given in (\ref{shellingorder2}) is a shelling order and the spanning simplices are $\tau_{i+10s'}^1$
	and $\tau_{i+10s'}^2$. The result follows by Proposition \ref{shellable}. 
	The case $5t = 3s$ and $n = 4t$, follows by symmetry.
	
\end{proof}
We need the following lemma to prove Theorem \ref{THEOREM3} (C).

\begin{lem} \label{collapse1} Let $ s, t \in \{1, \ldots, n-1\}$ such that $2s,2(s+t), 3s+t, 3s-t, 4s \not\equiv 0 \ ( \text{mod} \ n)$. Then 
	for each $k \in [n]$, $(\{s+k, n-s+k\}, N(k))$ is a collapsible pair in $\N(C_n(s, t))$.
\end{lem}
\begin{proof}
	If there exists $x \in [n], x \neq k$ such that $\{s+k, n-s+k\} \subset N(x)$,
	then $x \in \{2s+k, s+t+k,n-t+s+k\} \cap \{n-s+t+k, n-2s+k, n-t-s+k\}$.
	
	Since, $0 \not\equiv  3s-t, 2s+k \neq n-s+t+k$. Further, since   
	$2s+k = n-2s+k$ implies that $  4s \equiv 0$ and  $2s+k = n-t-s+k$ implies that $3s+t \equiv 0$, we conclude that $x \neq 2s+k$.
	
	$ s+t+k = n-s+t+k \implies 2s \equiv 0$ and   $s+t+k = n-2s+k \implies 3s+t \equiv 0$. Further,
	since $s+t+k = n-t-s+k \implies 0 \equiv 2(s+t)$, $x \neq s+t+k $.

	Since, $s \neq t$ and $3s-t \not\equiv 0$, we conclude that $n-t+s+k \neq n-s+t+k, n-2s+k, n-t-s+k$.   
	
	Thus there exists no $x \in [n]$ different from $k$ such that $\{s+k, n-s+k\} \subset N(x)$ and
	therefore $(\{s+k, n-s+k\}, N(k))$ is a collapsible pair.
\end{proof}

\begin{proof}[Proof of Theorem \ref{THEOREM3}(C)]
	Let $5s = 3t$ and  one of $6s$ or $\frac{14s}{3}$ is equal to $n$. There exists an integer $s'$ such that $s = 3s'$ and $t = 5s'$.

	\vspace{0.15 cm}
	
	\noindent{\bf{Case 1.}} $n = 6s$.
	
	In this case $n = 18s'$ and $ N(k) = \{3s'+k, 5s'+k, 13s'+k, 15s'+k\}$ for all  $k \in [n]$. Since $3s-t = n \implies \frac{4s}{3} = n$ and
	$3s+t = n \implies \frac{14s}{3} = n$, we see that
	$3s-t, 3s+t \neq n$. Further, $s, t < \frac{n}{2}$ implies that $3s-t, 3s+t \not\equiv n$. Using Lemma \ref{collapse1}, $N(k) \searrow \tau_k^1 = \{3s'+k, 5s'+k, 13s'+k\}$
	and $\tau_k^2 = \{15s'+k, 5s'+k, 13s'+k\}$ for all $k \in [n]$. Since, $\tau_{k+10s'}^1 = \{13s'+k, 15s'+k, 5s'+k\} = \tau_{k}^2$, $\N(C_n(s, t))$
	collapses to a subcomplex $\Delta$, where $M(\Delta) = \{\tau_{k}^1 \ | \ k \in [n]\}$.

	Let $k \in [n]$. Observe that $N(3s'+k) \cap N(5s'+k) = \{k, 8s'+k\}$ and $\{3s'+k, 5s'+k\}\nsubseteq \tau_{8s'+k}^1$. Hence 
	$(\{3s'+k, 5s'+k\}, \tau_k^1)$ is a collapsible pair in $\Delta$ and therefore 
	$\tau_k^1 \searrow \delta_{k}^1 = \{3s'+k, 13s'+k\}$ and $\delta_k^2 = \{5s'+k, 13s'+k\}$. Since 
	$\delta_{k+10s'}^1 = \{13s'+k, 5s'+k\} = \delta_k^2$, we see that 
	$\Delta$ collapses to a $1$-dimensional subcomplex $\Delta'$,
	with $M(\Delta') = \{\delta_k^1 \ | \ k \in [n]\}$. 
	Each vertex $x \in [n]$, belongs to $\delta_{5s'+x}^1 = \{8s'+x,x \}$
	and $\delta_{15s'+x}^1 = \{x, 10s'+x\}$. Since $\delta_{5s'+x}^1 \neq \delta_{15s'+x}^1$, 
	connected components of  $\Delta'$ cannot be contractible.
	Result follows from
	Remark \ref{remarkshe}.
	
	\vspace{0.15 cm}
	
	\noindent{{\bf Case 2.}} $n = \frac{14s}{3}$.
	
	In this case $n = 14s'$ and  
	$N(k) = \{3s'+k, 5s'+k, 11s'+k, 9s'+k\} \ \forall \ k \in [n]$.
	Since $3t-s = n \implies 4s = n$ and $3t+s = n \implies  6s = n$, we see that $3t-s, 3t+s \neq n$. Further, since
	$4t \neq n$, by Lemma \ref{collapse1}, $(\{5s'+k, 9s'+k\}, N(k))$ is a collapsible pair and therefore 
	$N(k) \searrow \tau_k^1 = \{5s'+k, 3s'+k, 11s'+k\}$ and $\tau_k^2 = \{9s'+k, 3s'+k, 11s'+k\}$. Since,
	$\tau_{k+6s'}^1 = \{11s'+k, 9s'+k, 3s'+k\} = \tau_{k}^2$, we conclude that 
	$\N(C_n(s, t))$ collapses to a subcomplex $\Delta$, where $M(\Delta) = \{\tau_k^1 \ | \ 
	k \in [n]\}$.
	
	$N(3s'+k) \cap N(5s'+k) = \{k, 8s'+k\}$ and   $\{3s'+k, 5s'+k\}
	\nsubseteq \tau_{8s'+k}^1$ implies that  $(\{3s'+k, 5s'+k\}, \tau_k^1)$ is a collapsible pair in 
	$\Delta$ and therefore $\tau_k^1 \searrow \delta_{k}^1 = \{3s'+k, 11s'+k\}$ and 
	$\delta_k^2 = \{5s'+k, 11s'+k\}$. Further, since $\delta_{8s'+k}^1 = \{11s'+k, 5s'+k\} = \delta_{k}^2$, 
	we conclude that $\Delta$ collapses to a $1$-dimensional subcomplex $\Delta'$, where $M(\Delta') = \{\delta_k^1 \ | \ k \in [n]\}$.
	
	Each vertex $x \in [n]$, belongs to $\delta_{3s'+x}^1 = \{6s'+x, x\}$ and $\delta_{11s'+x}^1 = \{x, 8s'+x\}$. 
	Since $\delta_{3s'+x}^1 \neq \delta_{11s'+x}^1$, the result follows from Remark \ref{remarkshe}.
	
	If $5t = 3s$ and  one of $6t$ or $\frac{14t}{3}$ is equal to $n$,  the result follows by symmetry. 
	
\end{proof}

\begin{proof}[Proof of Theorem \ref{THEOREM3}(D)]
	Let $5s = 3t$ and $n \neq 4s, 4t, 6s, \frac{14s}{3}$. Now, $3s - t = n \implies \frac{4s}{3} = n,$ which is not possible as $s < \frac{n}{2}$.
	Further, since  $3s+t = n \implies \frac{14s}{3} = n, 3t+s = n \implies 6s = n$ and  $3t-s = n \implies 
	4s = n$, we see  that $3s-t,3s+t, 3t+s, 3t-s \neq n$. The result follows from Theorem \ref{tori}.
	The case $5t = 3s$ and $n \neq 4s, 4t, 6t, \frac{14t}{3}$, follows by symmetry.

\end{proof}

\begin{proof}[Proof of Theorem \ref{THEOREM4}(A)]
	Let us first assume that one of the $3s-t, 3t-s, 3s+t \ \text{or}\ 3t+s$ is equal to $ n$.
	
	\vspace{0.15 cm}
	
	\noindent{\textbf{Case 1.}}  $3t - s = n$ or $3s-t = n$. 
	
	Assume that $3t-s = n$. We consider the following two cases.
	\begin{itemize}
		\item[(i)] $2s \neq t$.
		
		In this case, $3s+t, 3s-t, 4s \not\equiv n$. By Lemma \ref{collapse1}, 
		$N(k) \searrow \tau_k^{1} = \{s+k, t+k, n-t+k\}$ and $\tau_k^{2} = \{n-s+k, t+k, n-t+k\}$
		for all $k \in [n]$. Further, since  $\tau_{k}^2 = \tau_{t-s+k}^1$, 
		$\N(C_{n}(s,t))$ collapses to a subcomplex $\Delta$, with 
		$M(\Delta) = \{ \tau_{i}^1 \ | \ i \in [n]\}$. 
		It can be easily checked that  
		$N(s+k) \cap N(n-t+k) = \{k, 2t+k\}$. Since,
		$\{s+k, n-t+k\} \nsubseteq \tau_{2t+k}^1$, 
		$\tau_{k}^1 \searrow \delta_{k}^1 = \{s+k, t+k\}$ and $\delta_k^2 = \{n-t+k, t+k\}$ for all $k \in [n]$. Now, 
		$\delta_{t-s+k}^1 = \{t+k, n-t+k\} = \delta_{k}^2$ implies that 
		$\Delta$ collapses to a $1$-dimensional subcomplex $\Delta'$,
		where $M(\Delta') = \{\delta_{i}^1 \ | \ i \in [n]\}$.

		Each vertex $k \in [n]$ can belongs to  only  $\delta_j^{1}$ for $j \in \{s+k, t+k, n-s+k, n-t+k\}$. 
		Since, $\delta_{s+k}^1 = \{2s+k, s+t+k\},
		\delta_{t+k}^1 = \{s+t+k, 2t+k\}, \delta_{n-s+k}^1 = \{k, n-s+t+k\}, \delta_{n-t+k}^1 = \{n+s-t+k, k\}$,
		we observe that $k$ belongs to only  $ \delta_{n-s+k}^1$ and $ \delta_{n-t+k}^1$.
		Further, since $\delta_{n-s+k}^1 \neq \delta_{n-t+k}^1$ and
		$\Delta'$ is a $1$-dimensional complex, each connected component of $\Delta'$ is homotopy equivalent to $S^1$. 
		
		\item[(ii)] $2s = t$.
		
		In this case, $ n = 5s$ and $N(k) = \{s+k, 2s+k, 3s+k, 4s+k\} \ \forall \ k \in [n]$.  
		For   $i \in [n]$, let $\Gamma_i := \{i, s+i \ (\text{mod} \ n), 2s+i \ (\text{mod} \ n), 3s+i \ (\text{mod} \ n), 
		4s+i \ (\text{mod} \ n)\}$ be a $4$-simplex.
		It can be easily checked that $\Gamma_i \cap \Gamma_j = \emptyset \ \forall \ 1 \leq i \neq j \leq s.$ 
		The subcomplex of $\N(C_n(s,t))$ induced by the vertices 
		$k, k+s, k+2s, k+3s$ and  $k+4s$, which is equal to $N(k) \cup N(k+s) \cup N(k+2s) \cup N(k+3s) \cup N(k+4s)$
		is Bd$(\Gamma_k)$. We conclude that 
		$\N(C_n(s, t)) \cong \bigsqcup\limits_{k \in \{ 1, 2, \ldots, s\}} Bd(\Gamma_k)$. Since $Bd(\Gamma_k)$ is homeomorphic
		to $S^3$, the result follows. 
		
		The case $3s-t= n$ follows from symmetry.
	\end{itemize}

	\noindent{\textbf{Case 2.}} $3s-t, 3t-s \neq n$, {\it i.e.}, $3t+s = n$ or $3s+t = n$.
	
	Assume that $3t+s = n$.
	\begin{itemize}
		\item[(i)] $2t \neq s$.
		
		In this case,  $3s+t, 3s-t, 4s \not\equiv n$. From Lemma \ref{collapse1},  $N(k) \searrow \tau_k^1 = \{s+k, t+k, n-t+k\}$ 
		and $\tau_k^2 = \{n-s+k, t+k, n-t+k\}$ for all $k \in [n]$. Further, since  $\tau_k^2 = \tau_{2t+k}^1$, 
		$\N(C_n(s, t))$ collapses to a subcomplex $\Delta$, where $M(\Delta) = \{\tau_k^1 \ | \ k \in [n]\}$.
		
		It can be easily verified that $N(s+k) \cap N(t+k) = \{k, s+t+k\}$. Since,
		$\{s+k, t+k\} \nsubseteq \tau_{s+t+k}^1 = \{2s+t+k,s+2t+k, s+k \}$, $\tau_k^1 \searrow \delta_k^1 = \{s+k, n-t+k\}$ and 
		$\delta_k^2 = \{t+k, n-t+k\}$ for all $k \in [n]$. Further, since
		$\delta_{2t+k}^1 = \{2t+s+k, t+k \} = \{n-t+k, t+k\} = \delta_{k}^2$, 
		$\Delta$ collapses to a $1$-dimensional subcomplex $\Delta'$, where 
		$M(\Delta') = \{\delta_i^1 \ | \ i \in [n]\}$.

		Each vertex $k \in [n]$ can belong to  only $\delta_j^{1}$ for $j \in \{s+k, t+k, n-s+k, n-t+k\}$.
		Since, $\delta_{s+k}^1 = \{2s+k, n-t+s+k\}, \delta_{t+k}^1 = \{s+t+k, k\}, 
		\delta_{n-s+k}^1 =  \{k, 2t+k\}$ and $
		\delta_{n-t+k}^1= \{n-t+s+k, t+s+k\}$, we conclude that  $k$ belongs to only $ \delta_{t+k}^1$ and $\delta_{n-s+k}^1$.
		Since $\Delta'$ is a $1$-dimensional complex, each connected component of $\Delta$ is homotopy equivalent to $S^1$.
		
		\item[(ii)] $2t = s$.

		Since $n = 3t+s$ and $s= 2t, n = 5t$.  For each  $k \in [n]$, let $\Gamma_k := \{k, t+k (\text{mod} \ n), 2t+k (\text{mod} \ n), 
		3t+k(\text{mod} \ n), 4t+k (\text{mod} \ n)\}$ be a $4$-simplex.
		It can be easily verified that $\Gamma_i \cap \Gamma_j = \emptyset \ \forall \ 1 \leq i \neq j \leq t$ and 
		the subcomplex of $\N(C_n(s,t))$ induced by the vertices 
		$k, k+t, k+2t, k+3t$ and  $k+4t$,  is equal to $N(k) \cup N(k+t) \cup N(k+3t) \cup N(k+3t) \cup N(k+4t) = Bd(\Gamma_k)$. We conclude that 
		$\N(C_n(s, t)) \cong \bigsqcup\limits_{k \in \{ 1, 2, \ldots, t\}} Bd(\Gamma_k)$.

	\end{itemize}

	The case $3s+t = n$ follows by symmetry. 
	
\end{proof}

Proof of Theorem \ref{THEOREM4} (B) follows from Theorem \ref{gerland} and Proof of Theorem \ref{THEOREM4} (C) follows from  Theorem \ref{tori}.

\begin{thm}\label{gerland}
	Let $s, t \in \{1, 2, \ldots, \lfloor \frac{n}{2}\rfloor\}$ such that $2s, 2t, 2(s+t), 3s+t, 3s-t, 4s \neq n$, $3s \neq t, 3t \neq s$. If $4t = n$, then each connected 
	component of $\N(C_n(s,t))$ is homotopy equivalent to a garland of  the $2$-dimensional spheres $S^2$.
\end{thm}

\begin{proof} 
	By Lemma \ref{collapse1}, $\N(C_n(s,t))$ collapses to a $2$-dimensional subcomplex $\Delta$, whose maximal simplices
	$M(\Delta) = \{\tau_{k}^i \ | \ i \in \{1,2\}, k \in [n]\}$, where 
	$\tau_{k}^1 = \{s+k, t+k, n-t+k\}$ and $\tau_k^2 = \{n-s+k, t+k, n-t+k\}$.
	For any integer  $i$, 
	let $\Gamma_i := \{t+i \ (\text{mod} \ n), n-t+i \ (\text{mod} \ n), s+i \ (\text{mod} \ n), s+2t+i \ (\text{mod} \ n)\}$ be a $3$-simplex.
	Clearly, the simplicial complexes
	$\tau_k^1 \cup \tau_{s+t+k}^2 \cup  \tau_{2t+k}^1  \cup \tau_{n-t+s+k}^2 = Bd(\Gamma_k)$ and 
	$ \tau_{n-s-t+k}^1 \cup \tau_{k}^2 \cup \tau_{n-s+t+k}^1 \cup \tau_{2t+k}^2 = 
	Bd(\Gamma_{n-s-t+k})$. Hence, each $1$-simplex of $\Delta$ is part of Bd$(\Gamma_i)$ for some $i \in [n]$.
	Thus $\Delta = \bigcup\limits_{1 \leq i \leq n} Bd(\Gamma_i)$.
	
	In the rest of the proof, if we write an integer $x$ as a vertex of $ \Gamma_i$ for some $i$, then
	it is  understood that we are taking $x (\text{mod} \ n)$.  
	Now $Bd(\Gamma_{n-s-t+k}) \cap Bd(\Gamma_{k}) = \{t+k , n-t+k\}$ and
	$Bd(\Gamma_k) \cap Bd(\Gamma_{k+s+t}) = \{s+k, s+2t+k\}$, {\it i.e.}, $Bd(\Gamma_k)$
	shares a common $1$-simplex $\{t+k , n-t+k\}$ with $Bd(\Gamma_{n-s-t+k})$ and a common $1$-simplex
	$\{s+k, s+2t+k\}$ with $\Gamma_{s+t+k}$.
	
	To show that each component of $\Delta$ is homotopy equivalent to a garland of the $2$-dimensional spheres, 
	it is enough to show that each vertex $k$ belongs to exactly the boundaries of  two $3$-simplices.
	Let $x \in [n]$. There exists $k \in [n]$ such that $x \in Bd(\Gamma_k)$, {\it i.e.},
	$x \in \{t+k, n-t+k, s+k, s+2t+k\}$. It is clear from the above discussion that
	there exists $k' \neq k$ such that $x \in Bd(\Gamma_{k'})$, {\it i.e.}, $x \in \{t+k', n-t+k', s+k', s+2t+k'\}$. Since 
	$n = 4t$,
	$\Gamma_{k+2t} = \Gamma_{k} \ \forall \ k \in [n]$.
	Let $x = t+k$. 
	If $x \equiv t+k' $, then $k \equiv k'$, a contradiction. If $x \equiv n-t+k'$, then
	$k' \equiv k-2t$. But $\Gamma_{k-2t} = \Gamma_{k}$.
	Since,  $t+k \equiv s+k'$ implies that $k' \equiv t-s+k$ and  $t+k \equiv s+2t+k'$ implies that  
	$k' \equiv k-t-s$, we conclude that $x \in Bd(\Gamma_{k})$ and $Bd(\Gamma_{n-s-t+k})$ only.
	By a similar  argument  as the one above, we can easily verify that, 
	$n-t+k \in Bd(\Gamma_{k}) \cap Bd(\Gamma_{n-s-t+k})$,  $s+k \in Bd(\Gamma_{k}) \cap Bd(\Gamma_{s+t+k})$ and 
	$s+2t+k \in Bd(\Gamma_k) \cap Bd(\Gamma_{s+t+k})$ only.

	Thus $x$  belongs to either $Bd(\Gamma_k) \cap Bd(\Gamma_{n-s-t+k})$ or $Bd(\Gamma_{k}) \cap Bd(\Gamma_{s+t+k})$ only.

\end{proof}

A {\it $d$-dimensional pseudo manifold} is a pure $d$-dimensional simplicial complex
such that every ($d-1$)-simplex is a face of exactly two $d$-simplices.
A (topological) {\it $n$-manifold} is a hausdorff space $X$ such that every point  $x \in X$,
has a neighborhood which is homeomorphic to   $\mathbb{R}^n$. A $2$-manifold is called a {\it surface}.

\begin{thm} \label{tori}
	Let $s, t \in \{1, 2, \ldots, \lfloor \frac{n}{2}\rfloor\}$ such that $2s, 2t, 2(s+t), 3s+t, 3t+s, 3s-t, 3t-s, 4s, 4t \neq n, 3s \neq t$ and $ 3t\neq s$. Then each connected component 
	of $\N(C_{n}(s, t))$ is homotopy equivalent  to connected sum of tori. 
\end{thm}

We recall the following result to prove Theorem \ref{tori}.

\begin{prop} \label{universal}(Theorem 3A.3, \cite{h}) \\
	If $C$ is a chain complex of free abelian groups, then there exist short exact sequences
	\begin{center}
		$0 \longrightarrow H_n(C; \mathbb{Z}) \otimes \mathbb{Z}_2 \longrightarrow H_n(C; \mathbb{Z}_2) \longrightarrow$
		Tor$(H_{n-1}(C; \mathbb{Z}), \mathbb{Z}_2) \longrightarrow 0$
	\end{center}
	for all n and these sequences split.
	
\end{prop}

\begin{proof}[Proof of Theorem \ref{tori}]
	By Lemma  \ref{collapse1}, 
	$\N(C_{n}(s,t))$ collapses to a subcomplex $X$, with $M(X) = \{\tau_k^i \ | \ i \in \{1,2\}, k \in [n]\}$, 
	where $\tau_{k}^1 = \{s+k, t+k, n-t+k\}$  and $\tau_{k}^2 = \{n-s+k, t+k, n-t+k\}$.
	
	\begin{claim} \label{manifold}
		$X$ is a $2$-dimensional pseudo manifold.
	\end{claim}
	\begin{proof}[Proof of Claim \ref{manifold}]
		It is enough to show that any $1$-dimensional face of
		$\tau_{k}^1$ and $\tau_k^2$, {\it i.e.}, $\sigma_1^k = \{t+k, n-t+k\}, \sigma_2^k = \{s+k, t+k\},
		\sigma_3^k = \{n-t+k, n-s+k\}, \sigma_4^k = \{t+k, n-s+k\}$ and  $ \sigma_5^k = \{s+k, n-t+k\}$ 
		are faces of exactly two $2$-dimensional simplices of $X$.
		
		Clearly, $\sigma_1^k$ is a face of $\tau_k^1$ and $\tau_k^2$. By Lemma \ref{collapse1}, $(\sigma_1^k, N(k))$ is a collapsible pair and therefore  $\sigma_1^k$ 
		is not a face of any $\tau_{k'}^1$ or $\tau_{k'}^2$ for all $k' \neq k.$

			\begin{figure}[H]
			\begin{subfigure}[]{0.4\textwidth}
				\vspace{0.3 cm}
				\centering
				\begin{tikzpicture}
				[scale=0.3, vertices/.style={draw, fill=black, circle, inner sep=0.5pt, minimum size = 0pt,}]
				
				\node[vertices, label=below:{$\scriptstyle s+2t+k$}] (a) at (7, 1) {};
				\node[vertices, label=below:{$\scriptstyle 3t+k$}] (b) at (13,1) {};
				\node[vertices, label=left:{$\scriptstyle s+k$}] (c) at (1,7) {};
				\node[vertices, label=below:{$\scriptstyle x = t+k$}] (d) at (7,7) {};
				\node[vertices, label=right:{$\scriptstyle 2t-s+k$}] (e) at (13,7) {};
				\node[vertices, label=left:{$\scriptstyle n-t+k$}] (f) at (1,13) {};
				\node[vertices, label=above:{$\scriptstyle n-s+k$}] (g) at (7,13) {};

				\foreach \to/\from in
				{ a/b, a/d,a/c,b/d,b/e, c/d,c/f,d/e,d/f,d/g,f/g,e/g} \draw [-] (\to)--(\from);

				\path (4.8,1) -- (4.8,7) node[pos=0.7, text=blue]  {$\scriptstyle \tau_{k+t+s}^2$};
				\path (9,1) -- (9,7) node[pos=0.3, text=blue] {$\scriptstyle \tau_{k+2t}^1$};
				\path (11,1) -- (11,7) node[pos=0.7, text=blue] {$\scriptstyle \tau_{k+2t}^2$};
				\path (2.8,7) -- (2.8,13) node[pos=0.3, text=blue] {$\scriptstyle \tau_k^1$ };
				\path (4.8,7) -- (4.8,13) node[pos=0.7, text=blue] {$\scriptstyle \tau_{k}^2$};
				
				\path (9,7) -- (9,13) node[pos=0.3, text=blue] {$\scriptstyle \tau_{k+t-s}^1$};

				\end{tikzpicture}
				\caption{}\label{surface1}
			\end{subfigure}
			\begin{subfigure}[]{0.4\textwidth}
				\centering
				\begin{tikzpicture}
				[scale=0.34, vertices/.style={draw, fill=black, circle, inner sep=0.5pt, minimum size = 0pt, }]
				\tikzset{edge/.style = {->,> = latex'}}

				\node[vertices, label=below:{$\scriptstyle s+k$}] (a) at (7, 7) {};
				\node[vertices, label=below:{}] (b) at (13,1) {};
				\node[vertices, label=left:{}] (c) at (1,13) {};
				\node[vertices, label=above:{$\scriptstyle n-t+k$}] (d) at (7,13) {};
				\node[vertices, label=right:{}] (e) at (7,19) {};
				\node[vertices, label=below:{$\scriptstyle t+k$}] (f) at (13,7) {};
				\node[vertices, label=right:{$\scriptstyle n-s+k$}] (g) at (13,13) {};
				\node[vertices, label=above:{}] (h) at (19,7) {};

				\foreach \to/\from in
				{ a/b, a/c,  b/f, c/d, d/e, e/g,  g/h, f/h} \draw [-] (\to)--(\from);

				\path (10.6,1) -- (10.6,7) node[pos=0.7, text=blue] {$\scriptstyle \tau_{k+s+t}^2$};
				
				\path (8.6,8) -- (8.6,15) node[pos=0.95, text=blue] {$\scriptstyle \tau_{k-s-t}^1$};
				
				\path (5,7) -- (5,13) node[pos=0.65, text=blue] {$\scriptstyle \tau_{k-t+s}^2$ };
				\path (11,7) -- (11,13) node[pos=0.6, text=blue] {$\scriptstyle \tau_{k}^2$};
				
				\path (8,7) -- (8,13) node[pos=0.3, text=blue] {$\scriptstyle \tau_{k}^1$};
				\path (14.5,7) -- (14.5,13) node[pos=0.3, text=blue] {$\scriptstyle \tau_{k+t-s}^1$};

				\path (10,7) -- (10,13) node[pos=0.35, text=red] {$\scriptstyle \sigma_1^k$};
				\path (7,6.5) -- (13,6.5) node[pos=0.5, text=red] {$\scriptstyle \sigma_2^k$};
				\path (7,12.5) -- (13,12.5) node[pos=0.5, text=red] {$\scriptstyle \sigma_3^k$};
				\path (13.5,7) -- (13.5,13) node[pos=0.6, text=red] {$\scriptstyle \sigma_4^k$};
				\path (7,7) -- (7,13) node[pos=0.35, text=red] {$\scriptstyle \sigma_5^k$};

				\draw[edge] (7.5,10.5) to[bend left=85] (9,8); 
				
				\draw[edge] (10.5,12) to[bend left=85] (12,9.5); 
				
				\draw[edge] (14.3,10.2) to[bend left=85] (15.8,7.7); 
				
				\draw[edge] (10.5,6) to[bend left=85] (12,3.5); 
				
				\draw[edge] (8.5,16) to[bend left=85] (10,13.5); 
				
				\draw[edge] (4.8,12) to[bend left=85] (6.3,9.5); 
				
				\draw[edge] (d) to (f); 
				\draw[edge] (f) to (a);
				\draw[edge] (d) to (g);
				\draw[edge] (g) to (f);
				\draw[edge] (a) to (d);
				
				\end{tikzpicture}
				\caption{}\label{surface2}
			\end{subfigure}
			\caption{} \label{surface}
		\end{figure}

			$\sigma_2^k \subset \tau_{k+s+t}^2 = \{k+t, s+2t+k, s+k\}$. Since 
		$s, t < \frac{n}{2}, s+2t+k \equiv n-t+k$
		implies that $3t+s = n$, which is not possible. Hence $\tau_{k+s+t}^2 \neq \tau_{k}^1$. Further, $3s-t $ and $3t-s \neq n$ implies that 
		$N(s+k) \cap N(t+k) = \{k, s+t+k\}$. Since, $\sigma_2^k \nsubseteq \tau_{s+t+k}^1$, 
		$\sigma_2^k$ is a face of $\tau_{k}^1$ and $\tau_{k+s+t}^2$ only.
		Since $n-t+k \equiv s+ (k-s-t)$ and  $n-s+k \equiv t+(k-s-t)$, 
		we see that, $\sigma_3^k$ is a face of $\tau_{k-s-t}^1$ and $\tau_{k}^2$ only.

		$\sigma_4^k \subset \tau_{k+t-s}^1 = \{t+k, 2t-s+k, n-s+k\}$ and $\sigma_4^k \nsubseteq \tau_{k+t-s}^2$. Since $2t-s+k \equiv n-t+k$ 
		implies that $3t-s = n$, we see that $\tau_{k+t-s}^1 \neq \tau_{k}^2$. Further, $2(s+t), 3s+t, 3t+s \neq n$ implies that 
		$N(t+k)  \cap N(n-s+k) = \{k, k-t+s\}$. Thus $\sigma_4^k$ is a face of $\tau_k^2$ and $\tau_{k+t-s}^1$ only.
		Since $s+k = t+(k-t+s)$ and $n-t+k = n-s+(k-t+s)$, we see that, $\sigma_5^k$ is a face of $\tau_{k-t+s}^2$ and $\tau_k^1$ only.
	\end{proof}

	\begin{claim} \label{surfacemanifold}
		$| |X| |$, the geometric realization of $X$,  is a surface, {\it i.e.}, a $2$-manifold. 
	\end{claim}
	\begin{proof}[Proof of Claim \ref{surfacemanifold}]
		Let $x \in$ $| |X| |$. Since, $X$ is a pseudomanifold, if $x$ belongs to
		interior of some $1$-simplex or $2$-simplex,
		then we can easily construct an open neighborhood of $x$, homeomorphic to
		$\mathbb{R}^2$. Assume $x$ is a $0$-simplex of $X$.
		Without loss of generality, we can  assume that  $x = t+k$, for some $k \in [n]$.
		Then $x \in \tau_{k}^1, \tau_{k}^2, \tau_{k+t-s}^2, \tau_{k+2t}^2, \tau_{k+2t}^1$ and
		$\tau_{k+t+s}^2$(see Figure \ref{surface1}). Using the fact that $2(s+t), 3t-s, 3t+s, 4t \not\equiv 0$, it can be easily checked that
		$|\{s+k, n-t+k, n-s+k, 2t-s+k, 3t+k, s+2t+k\}| = 6$. 
		Since $N(x) = \{k, k+t-s, k+2t, k+s+t\}$, $x \notin \tau_{m}^1 \ \text{or} \ \tau_{m}^2$
		for all $m \neq k, k+t-s, k+2t, k+s+t$. Hence  we can easily construct a neighborhood of $x$ in $| |X| |$,
		which is homeomorphic to $\mathbb{R}^2.$ Thus $X$ is a surface.
	\end{proof}

	We now define the orientations for simplices of $X$. For any oriented simplex $\sigma$ we write $+ \sigma$,
	if it is positively oriented and  $- \sigma$, otherwise. We denote any positively oriented $2$-simplex
	with vertex set $\{a,b,c\}$ by $\langle a, b , c \rangle$ and 
	$1$-simplex with vertex set $\{a,b\}$ by $\langle a, b \rangle$. 
	
	For any $k \in [n]$, we define $+ \tau_k^1 = \langle s+k, n-t+k, t+k \rangle$ and
	$+ \tau_k^2  = \langle n-t+k, n-s+k, t+k \rangle$. Further, we define
	$+ \sigma_1^k = \langle n-t+k, t+k \rangle, + \sigma_2^k = \langle t+k, s+k \rangle, +\sigma_3^k = \langle n-t+k,
	n-s+k \rangle,
	+ \sigma_4^k = \langle n-s+k, t+k \rangle$ and $+ \sigma_5^k = \langle s+k, n-t+k \rangle.$
	Observe that any $1$-simplex of $X$ is equal to $\sigma_i^k$
	for some $k \in [n]$ and $i \in \{1,2,3,4,5\} $.
	
	Let $C = (C_i, \partial_i)$ be the  simplicial chain complex of $X$ with coefficients in $\mathbb{Z}_2$.
	Since, there are $n$ $0$-simplices in $X, C_0 \cong \mathbb{Z}_2^n$. Further,
	since there are $2n$ simplices of dimension $2$ and $X$ is a pseudo manifold, $C_1 \cong \mathbb{Z}_2^{3n}$
	and $C_2 \cong \mathbb{Z}_2^{2n}$. Since there is no simplex of any other dimension in $X$, $C_i = 0 $
	for all $i \neq 1,2,3$. Thus 
	\begin{center}
		$C  =  0 \longrightarrow \mathbb{Z}_2^{2n} 
		\overset{\partial_2}{\longrightarrow} \mathbb{Z}_2^{3n} \overset{\partial_1}{\longrightarrow} \mathbb{Z}_2^n
		\overset{\partial_0}{\longrightarrow} 0$
	\end{center}
	Let $p$ be the number of connected components of $X$. It is well known that
	$H_0(X; \mathbb{Z}_2) \cong \mathbb{Z}_2^{p}$.
	Since $p \geq 1$, Rank$(\partial_1) \leq n-1$. If
	Ker $\partial_1 \cong \mathbb{Z}_2^{r}$, then  $r \geq 3n-n+1 = 2n+1$. Since Rank$(\partial_2)$
	can be at most $2n$, $H_1(X; \mathbb{Z}_2) \neq 0$. 
	From Proposition \ref{universal}, $H_1(X; \mathbb{Z}_2) \cong   H_1(X; \mathbb{Z}) \otimes \mathbb{Z}_2 $
	$\oplus$Tor$(H_{0}(X; \mathbb{Z}), \mathbb{Z}_2)$.  Since
	$H_0(X; \mathbb{Z})$ $  \cong \mathbb{Z}^p$, Tor$(H_0(X); \mathbb{Z}_2) = 0$.
	So $H_1(X; \mathbb{Z}) = 0$, implies that $H_1(X; \mathbb{Z}_2) = 0$, which is a contradiction.
	Hence $H_1(X; \mathbb{Z}) \neq 0.$

	Let $D = (D_i, d_i)$ be the simplicial chain complex of $X$ with $\mathbb{Z}$  coefficients. Then 
	\begin{center}
		$D =  0 \longrightarrow \mathbb{Z}^{2n} 
		\overset{d_2}{\longrightarrow} \mathbb{Z}^{3n} \overset{d_1}{\longrightarrow} \mathbb{Z}^n
		\overset{d_0}{\longrightarrow} 0$
	\end{center}
	
	Let $c = \sum\limits_{k \in [n]} ( (+\tau_k^1) + (+\tau_k^2))$ be a $2$-chain.  
	It can be easily verified (see Figure \ref{surface2}) that 
	$\sigma_1^k$ has $+$ve orientation in $\tau_k^1$ and $-$ve orientation in $\tau_k^2$.
	$\sigma_2^k$ has $+$ve orientation in $\tau_k^1$ and $-$ve in $\tau_{k+s+t}^2$.
	Similarly $\sigma_3^k$ has $+$ve orientation in $\tau_k^2$ and $-$ve in $\tau_{k-t-s}^1$.
	The simplex $\sigma_4^k$ having $+$ve orientation in $\tau_k^2$ and $-$ve in $\tau_{k+t-s}^1$,
	and $\sigma_5^k$ has $+$ve orientation in $\tau_k^1$ and $-$ve in $\tau_{k-t+s}^2$.
	Since $X$ is a pseudo manifold, each $1$-simplex of $X$
	will occur twice in  $d_2(c)$, once with $+$ve sign and once with $-$ve sign.   Hence $d_2(c) = 0$.
	Since $\tau_k^i \neq 0$ in $D_2 = \mathbb{Z}^{2n}, 0 \neq  c  \in $ Ker $d_2$. Hence
	$H_2(X; \mathbb{Z}) \neq 0$.
	
	From Claim \ref{surfacemanifold}, each component of $X$ is a compact surface.
	From the classification of surfaces, we know that  any connected compact surface is
	homeomorphic either to $S^2$, to a connected sum of tori or to a connected sum of projective planes.
	
	Since $H_1(X; \mathbb{Z}) \neq 0$, the connected components of $X$ cannot be
	homeomorphic to $S^2$. Further, since $H_2(S; \mathbb{Z}) = 0$ for any non orientable surface $S$, 
	the connected components of $X$ cannot be homeomorphic to a connected sum of projective planes.
	Hence each connected component of $X$ is homeomorphic to a connected sum of tori.
	
\end{proof}

Let $\Delta$ be a simplicial complex. An {\it $m$-path} in $\Delta$ is a sequence $\sigma_1 \ldots \sigma_t, t \geq 2$
of $m$-simplices  such that $\sigma_i$ and $\sigma_{i+1}$ have a common $(m-1)$-dimensional face, for all $ 1\leq i \leq t-1$.
Further, if $\sigma_1$ and $\sigma_t$ have a common $(m-1)$-dimensional face, then it is said to be closed $m$-path.

	\begin{figure}[H]

	\centering
	\begin{tikzpicture}
	[scale=0.35, vertices/.style={draw, fill=black, circle, inner sep=0.5pt, minimum size = 0pt, }]
	
	\node[vertices, label=below:{$\scriptstyle s+k$}] (m) at (1,1) {};
	\node[vertices, label=below:{$\scriptstyle s+k+u$}] (n) at (5.5,1) {};
	\node[vertices, label=above:{$\scriptstyle n-t+k$}] (o) at (1,4) {};
	\node[vertices, label=above:{$\scriptstyle n-t+k+u$}] (p) at (5.5,4) {};
	\node[vertices, label=above:{$\scriptstyle n-t+k+2u$}] (q) at (10.5,4) {};
	\node[vertices, label=below:{$\scriptstyle s+k+2u$}] (r) at (10.5,1) {};
	\node[vertices, label=above:{$\scriptstyle n-t+k+(p-1)u$}] (s) at (16,4) {};
	\node[vertices, label=above:{$\scriptstyle n-t+k+pu=n-t+k$}] (t) at (24,4) {};
	\node[vertices, label=below:{$\scriptstyle s+k+(p-1)u$}] (u) at (16,1) {};
	\node[vertices, label=below:{$\scriptstyle s+k+pu=s+k$}] (v) at (24,1) {};

	\node[vertices, label=left:{$\scriptstyle n-t+k+qv=n-t+k$}] (a) at (10.5,8) {};
	\node[vertices, label=right:{$\scriptstyle n-s+k+qv=n-s+k$}] (b) at (15,8) {};
	\node[vertices, label=left:{$\scriptstyle n-t+k+(q-1)v$}] (c) at (10.5,12.5) {};
	\node[vertices, label=right:{$\scriptstyle n-s+k+(q-1)v$}] (d) at (15,12.5) {};
	\node[vertices, label=left:{$\scriptstyle n-t+k+2v$}] (e) at (10.5,17) {};
	\node[vertices, label=right:{$\scriptstyle n-s+k+2v$}] (f) at (15,17) {};
	\node[vertices, label=left:{$\scriptstyle n-t+k+v =s+k$}] (g) at (10.5,21.5) {};
	\node[vertices, label=right:{$\scriptstyle n-s+k+v = t+k$}] (h) at (15,21.5) {};
	\node[vertices, label=left:{$\scriptstyle n-t+k$}] (i) at (10.5,26) {};
	\node[vertices, label=right:{$\scriptstyle n-s+k$}] (j) at (15,26) {};

	\foreach \to/\from in
	{ o/m, o/p,o/n, m/n, n/p, p/q, p/r, q/r, r/n, s/u, s/t, u/v, t/v, s/v, a/b, a/c,b/d,c/d,e/f, f/h,e/g,g/h, g/i, h/j, i/j, b/c,f/g,h/i} \draw [-] (\to)--(\from);
	\foreach \to/\from in
	{ r/u, q/s, e/c,f/d} \draw [dashed] (\to)--(\from);
	
	font=\scriptsize

	\path (2,1) -- (2,4) node[pos=0.3, text=red]  {$\scriptstyle \tau_k^1$};
	\path (4,1) -- (4,4) node[pos=0.7, text=blue] {$\scriptstyle \tau_k^2$};
	
	\path (7,1) -- (7,4) node[pos=0.3, text=red]  {$\scriptstyle \tau_{k+u}^1$};
	\path (9,1) -- (9,4) node[pos=0.7, text=blue] {$\scriptstyle \tau_{k+u}^2$};
	
	\path (18.3,1) -- (18.3,4) node[pos=0.3, text=red]  {$\scriptstyle \tau_{k+(p-1)u}^1$};
	\path (22,1) -- (22,4) node[pos=0.7, text=blue] {$\scriptstyle \tau_{k+(p-1)u}^2$};

	\path (11.5,21.5) -- (11.5,26) node[pos=0.3, text=red]  {$\scriptstyle \tau_k^1$};
	\path (13.5,21.5) -- (13.5,26) node[pos=0.7, text=blue] {$\scriptstyle \tau_k^2$};

	\path (12,17) -- (12,21.5) node[pos=0.3, text=red]  {$\scriptstyle \tau_{k+v}^1$};
	\path (13.5,17) -- (13.5,21.5) node[pos=0.7, text=blue] {$\scriptstyle \tau_{k+v}^2$};
	
	\path (12.4,8) -- (12.4,12.5) node[pos=0.15, text=red]  {$\scriptstyle \tau_{k+(q-1)v}^1$};
	\path (13.2,8) -- (13.2,12.5) node[pos=0.85, text=blue] {$\scriptstyle \tau_{k+(q-1)v}^2$};
	
	\end{tikzpicture}
	\caption{}\label{figuretorus1}
\end{figure}

\begin{proof}[Proof of Theorem \ref{special}]
	Using Lemma \ref{collapse1}, $\N(C_n(s,t))$ collapses to a subcomplex
	$X$ with $M(X) = \{\tau_k^i \ | \ i \in \{1,2\}, k \in  [n]\}$, where $\tau_k^1 = \{s+k, t+k, n-t+k\}$ and $\tau_k^2 = \{n-s+k, t+k, n-t+k\}$.
	Let $u = t-s$ and $v = t+s$. Let $k \in [n]$.
	Since, $s+u+k = t+k$ and $n-t+u+k =n-s+k$, we see that 
	$\tau_k^2$ and $\tau_{k+u}^1$
	have a common 1-simplex $\{n-s+k, t+k\}$. Further, $n-s+k+v=  t+k$ and $n-t+k+v = s+k$ implies that  $\tau_k^{1}$ and $ \tau_{k+v}^2$ have a
	common $1$-simplex $\{s+k, t+k,\}$. Since $\tau_k^1$ and $\tau_k^2$ have a common 1-simplex $\{t+k, n-t+k\}$, we conclude that $\tau_k^1 \tau_k^2$, 
	$ \tau_{k}^2 \tau_{k+u}^1$ and  $\tau_k^1 \tau_{k+v}^2$ are  $2$-paths 
	in $X$.

	\begin{itemize}
		\item[(i)] $s = \frac{p-q}{2}$ and $t = \frac{p+q}{2}$.
		
		In this case $q=t-s = u$ and $p=t+s=v$.
		Now, $\omega = \tau_{k}^1 \tau_k^2 \tau_{k+u}^1 \tau_{k+u}^2 \ldots $ $ \tau_{k+(p-1)u}^1 \tau_{k+(p-1)u}^2$ and
		$\Gamma = \tau_{k}^2 \tau_{k}^1 \tau_{k+v}^2 \tau_{k+v}^1 \ldots  \tau_{k+(q-1)v}^2 \tau_{k+(q-1)v}^1$ are 
		$2$-paths in $X$.  Since  $k+pu \ \text{and} \ k+qv \equiv k \ (\text{mod \ n})$,
		$\omega$ (see the horizontal rectangular strip of Figure \ref{figuretorus1})  and $\Gamma$
		(see the vertical rectangular strip of Figure \ref{figuretorus1}) are  closed paths in $X$. 
		Thus  the simplices of $X$ can be arranged in a rectangular grid of order $(q+1) \times (p+1)$ as 
		depicted in Figure \ref{figuretorus2}. Since $X$ has $2n$ $1$-dimensional simplices, to prove that this rectangular grid  gives a triangulation of  a torus, it is enough to show that
		there is no identification among the vertices other than that  shown in Figure \ref{figuretorus2}. Any vertex of
		$C_n(s,t)$ can be written as $n-t+k$ for some $k \in [n]$ and 
		therefore can be made as the left uppermost corner vertex $a_{1,1}$ of this grid.   Hence, it is enough to show that $n-t+k = a_{1,1} \neq a_{i,j}$ unless
		$i \in \{1, q+1\}$
		and $j \in \{1, p+1\}$, {\it i.e.}, $n-t+k$ lies only on the four corners of the grid. But, since $a_{i, j} = (i-1)p+n-t+k+(j-1)q$, $a_{i, j} = n-t+k$ implies that $(i-1)p+(j-1)q \equiv 0 \ (\text {mod \ n})$.
		Further, since $gcd(p,q) = 1$, this is only possible if  $i \in \{1, q+1\}$ and $j \in \{1, p+1\}$.

		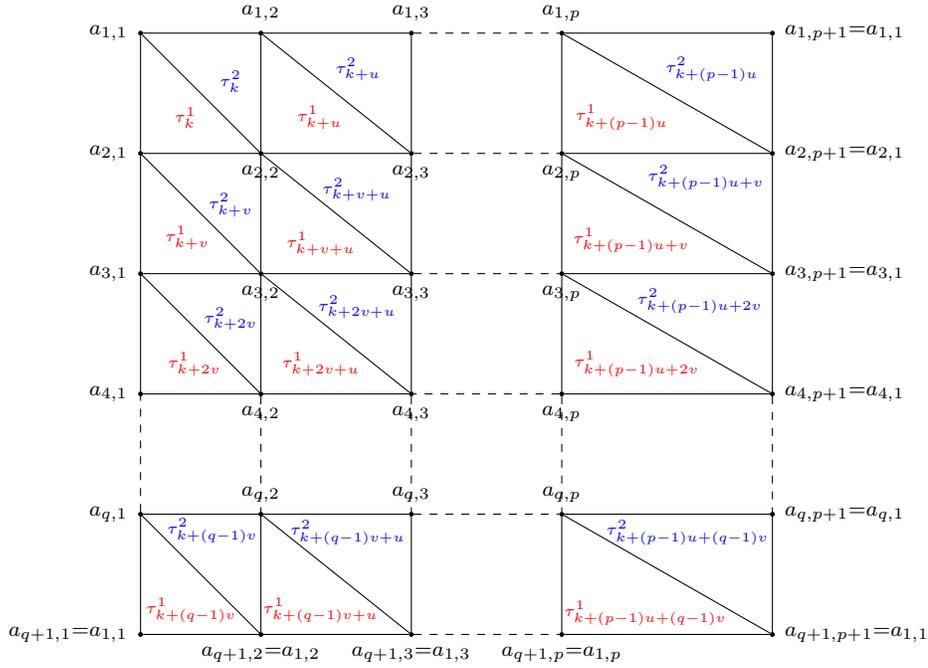
\begin{figure}[H]
			\centering
			\begin{tikzpicture}
			[scale=0.4, vertices/.style={draw, fill=black, circle, inner sep=0.5pt, minimum size = 0pt, }]
			
			\node[vertices, label=left:{$\scriptstyle a_{1,1}$}] (o) at (1,25) {};
			\node[vertices, label=above:{$\scriptstyle a_{1,2}$}] (p) at (5,25) {};
			\node[vertices, label=above:{$\scriptstyle a_{1,3}$}] (q) at (10,25) {};
			\node[vertices, label=above:{$\scriptstyle a_{1,p}$}] (s) at (15,25) {};
			\node[vertices, label=right:{$\scriptstyle a_{1,p+1} = a_{1,1}$}] (t) at (22,25) {};
			
			\node[vertices, label=below:{$\scriptstyle a_{2,p}$}] (u) at (15,21) {};
			\node[vertices, label=right:{$\scriptstyle a_{2,p+1} = a_{2,1}$}] (v) at (22,21) {};
			\node[vertices, label=left:{$\scriptstyle a_{2,1}$}] (m) at (1,21) {};
			\node[vertices, label=below:{$\scriptstyle a_{2,2}$}] (n) at (5,21) {};
			\node[vertices, label=below :{$\scriptstyle a_{2,3}$}] (r) at (10,21) {};
			
			\node[vertices, label=below:{$\scriptstyle a_{3,3}$}] (bbb) at (10,17) {};
			\node[vertices, label=below:{$\scriptstyle a_{3,p}$}] (w) at (15,17) {};
			\node[vertices, label=right:{$\scriptstyle a_{3,p+1} = a_{3,1}$}] (x) at (22,17) {};
			\node[vertices, label=left:{$\scriptstyle a_{3,1}$}] (k) at (1,17) {};
			\node[vertices, label=below:{$\scriptstyle a_{3,2}$}] (l) at (5,17) {};
			
			\node[vertices, label=below:{$\scriptstyle a_{4,p}$}] (y) at (15,13) {};
			\node[vertices, label=right:{$\scriptstyle a_{4,p+1} = a_{4,1}$}] (z) at (22,13) {};
			\node[vertices, label=left:{$\scriptstyle a_{4,1}$}] (i) at (1,13) {};
			\node[vertices, label=below:{$\scriptstyle a_{4,2}$}] (j) at (5,13) {};
			\node[vertices, label=below:{$\scriptstyle a_{4,3}$}] (aaa) at (10,13) {};

			\node[vertices, label=above:{$\scriptstyle a_{q,3}$}] (ccc) at (10,9) {};
			\node[vertices,label=left:{$\scriptstyle a_{q,1}$}] (e) at (1,9) {}; 
			\node[vertices, label=above:{$\scriptstyle a_{q,2}$}] (f) at (5,9) {};
			\node[vertices, label=above:{$\scriptstyle a_{q, p}$}] (g) at (15,9) {};
			\node[vertices, label=right:{$\scriptstyle a_{q,p+1} = a_{q,1}$}] (h) at (22,9) {};
			
			\node[vertices, label=below:{$\scriptstyle a_{q+1,3} = a_{1,3}$}] (ddd) at (10,5) {};
			\node[vertices, label=left: {$\scriptstyle a_{q+1,1} = a_{1,1}$}] (a) at (1,5) {};
			\node[vertices, label=below:{$\scriptstyle a_{q+1, 2} = a_{1,2}$}] (b) at (5,5) {};
			\node[vertices,label=below:{$\scriptstyle a_{q+1, p} = a_{1,p}$}] (c) at (15,5) {};
			\node[vertices,label=right:{$ \scriptstyle a_{q+1, p+1} = a_{1,1}$}] (d) at (22,5) {};

			\foreach \to/\from in
			{a/b,a/e,e/f,b/f,c/d,c/g,g/h,h/d,k/i, k/l, l/j, i/j, k/j, o/m, o/p,o/n, m/n, n/p, m/k, m/l, l/n, p/q, p/r, q/r, r/n, s/u, s/t, u/v, t/v, s/v, u/w,v/x, w/x, u/x, w/y, x/z, y/z, z/w, e/b, g/d, aaa/bbb, l/bbb, l/aaa, j/aaa, n/bbb, r/bbb, ccc/ddd, f/ddd, f/ccc,
				b/ddd} \draw [-] (\to)--(\from);
			\foreach \to/\from in
			{ddd/c,ccc/g, e/i,j/f, q/s, r/u, bbb/w, y/g, z/h, y/aaa, aaa/ccc} \draw [dashed] (\to)--(\from);
			
			\coordinate (a) at (1,21);
			\coordinate (aa) at (6,25);
			
			\coordinate (b) at (6,21);
			\coordinate (bb) at (9.2,25);
			
			\coordinate (c) at (14,21);
			\coordinate (cc) at (22,25);
			
			\coordinate (d) at (1,17);
			\coordinate (dd) at (6.2,21);

			\coordinate (e) at (14,17);
			\coordinate (ee) at (22,21);

			\coordinate (f) at (1,1);
			\coordinate (ff) at (6.5,4.5);

			font=\scriptsize

			\path (a) -- (aa) node[pos=0.3, text=red]  {$\scriptstyle \tau_k^1$};
			\path (a) -- (aa) node[pos=0.6, text=blue] {$\scriptstyle \tau_k^2$};
			
			\path (b) -- (bb) node[pos=0.3, text=red]  {$\scriptstyle \tau_{k+u}^1$};
			\path (b) -- (bb) node[pos=0.7, text=blue] {$\scriptstyle \tau_{k+u}^2$};
			
			\path (17,21) -- (17,25) node[pos=0.3, text=red]  {$\scriptstyle \tau_{k+(p-1)u}^1$};
			\path (20,21) -- (20,25) node[pos=0.7, text=blue] {$\scriptstyle \tau_{k+(p-1)u}^2$};
			
			\path (d) -- (dd) node[pos=0.3, text=red]  {$\scriptstyle \tau_{k+v}^1$};
			\path (4,17) -- (4,22) node[pos=0.46, text=blue] {$\scriptstyle \tau_{k+v}^2$};
			
			\path (15.5,17) -- (22,21) node[pos=0.28, text=red]  {$\scriptstyle \tau_{k+(p-1)u+v}^1$};
			\path (19.8,17) -- (19.8,21) node[pos=0.8, text=blue] {$\scriptstyle \tau_{k+(p-1)u+v}^2$};

			\path (2.7, 5) -- (2.7,9) node[pos=0.2, text=red]  {$\scriptstyle \tau_{k+(q-1)v}^1$};
			\path (3.4, 5) -- (3.4,10) node[pos=0.69, text=blue] {$\scriptstyle \tau_{k+(q-1)v}^2$};
			
			\path (2.8, 13) -- (2.8, 18) node[pos=0.2, text=red] {$\scriptstyle \tau_{k+2v}^1$};
			\path (4, 13) -- (4,18) node[pos=0.5, text=blue] {$\scriptstyle \tau_{k+2v}^2$};
			
			\path (17.5, 13) -- (17.5,18) node[pos=0.2, text=red] {$\scriptstyle \tau_{k+(p-1)u+2v}^1$};
			\path (19.6, 13) -- (19.6,18) node[pos=0.62, text=blue] {$\scriptstyle \tau_{k+(p-1)u+2v}^2$};
			
			\path (17.8,5 ) -- (17.8,9) node[pos=0.18, text=red] {$\scriptstyle \tau_{k+(p-1)u+(q-1)v}^1$};
			
			\path (19.2,5 ) -- (19.2,9) node[pos=0.85, text=blue] {$\scriptstyle \tau_{k+(p-1)u+(q-1)v}^2$};

			\path (7, 13 ) -- (7,18) node[pos=0.2, text=red] {$\scriptstyle \tau_{k+2v+u}^1$};
			
			\path (8.2, 13 ) -- (8.2,18) node[pos=0.58, text=blue] {$\scriptstyle \tau_{k+2v+u}^2$};
			
			\path (7, 17 ) -- (7,21) node[pos=0.25, text=red] {$\scriptstyle \tau_{k+v+u}^1$};
			
			\path (8.2, 17 ) -- (8.2,21) node[pos=0.7, text=blue] {$\scriptstyle \tau_{k+v+u}^2$};
			
			\path (7, 5 ) -- (7,9) node[pos=0.2, text=red] {$\scriptstyle \tau_{k+(q-1)v+u}^1$};
			\path (7.9, 5 ) -- (7.9,9) node[pos=0.85, text=blue] {$\scriptstyle \tau_{k+(q-1)v+u}^2$};
			
			\end{tikzpicture}
			\caption{$\scriptstyle u = t-s, v = t+s, a_{i,j} = (i-1)v+ n-t+k+(j-1)u, 1 \leq i \leq q+1, 1 \leq j \leq p+1$.} \label{figuretorus2}
		\end{figure}

		\item[(ii)]  $s = \frac{p^2-q}{2} $ and $t = \frac{p^2+q}{2}$.
		
		In this case, $v = s+t = p^2$ and $u = t-s= q$. Since $pq= n$, we have the closed
		$2$-paths $\tau_{k}^1 \tau_k^2 \tau_{k+u}^1 \tau_{k+u}^2 \ldots \tau_{k+(p-1)u}^1 \tau_{k+(p-1)u}^2$ and 
		$\tau_{k}^2 \tau_{k}^1 \tau_{k+v}^2 \tau_{k+v}^1 \ldots  \tau_{k+(q-1)v}^1 \tau_{k+(q-1)v}^1$.
		
		In this case also, we can arrange the simplices of $X$ in a rectangular grid of
		order $(q+1) \times (p+1)$, as depicted in Figure \ref{figuretorus2}. By an argument similar as of the case $(i)$,
		to prove that this rectangular grid  gives a triangulation of a torus, it is enough to show that
		$n-t+k = a_{1,1} \neq a_{i,j}$ unless $i \in \{1, q+1\}$ and $j \in \{1, p+1\}$.  
		But, since $a_{i, j} = (i-1)v+n-t+k+(j-1)u$, $a_{i, j} = n-t+k$ implies that $(i-1)v+(j-1)u \equiv 0 \ (\text {mod \ n})$.
		Since $v = p^2, u = q, (i-1)p^2+(j-1)q \equiv 0 \ (\text {mod \ n})$. This is only possible when
		$i \in \{1, q+1\}$ and $j \in \{1, p+1\}$ as $gcd(p,q)=1.$

	\end{itemize}
	
\end{proof}

\subsection*{Acknowledgements}
The author would like to thank anonymous  referees for their helpful  suggestions which led to a significant improvement in the presentation of this article.

\end{document}